\pdfoutput=1
 \documentclass[11pt]{amsart}
\usepackage{amsmath,amssymb}
\usepackage{graphicx}
\usepackage{amsfonts,amscd,latexsym}

\newcommand{\labell}[1] {\label{#1}}

\newcommand{\ts}{\textstyle}

\newcommand{\TL}{{\Tilde L}}
\newcommand{\TE}{{\Tilde E}}

\newcommand{\ocp}{{\ov{\C P}\,\!}}

\newcommand{\Rr}{{\mathcal R}}
\newcommand{\Ee}{{\mathcal E}}
\newcommand{\Ll}{{\mathcal L}}

\newcommand{\less}{{\smallsetminus}}
\newcommand{\Ff}{{\mathcal F}}

\newcommand{\un}{\underline}

\newcommand{\p}{{\partial}}
\newcommand{\al}{{\alpha}}

\newcommand{\be}{{\beta}}

\newcommand{\Om}{{\Omega}}
\newcommand{\om}{{\omega}}
\newcommand{\eps}{{\varepsilon}}
\newcommand{\ve}{{\epsilon}}
\newcommand{\de}{{\delta}}
\newcommand{\De}{{\Delta}}
\newcommand{\ga}{{\gamma}}

\newcommand{\io}{{\iota}}
\newcommand{\ka}{{\kappa}}
\newcommand{\la}{{\lambda}}

\newcommand{\si}{{\sigma}}

\newcommand{\Cc}{{\mathcal C}}
\newcommand{\Hh}{{\mathcal H}}

\newcommand{\Uu}{{\mathcal U}}

\newcommand{\Vv}{{\mathcal V}}

\newcommand{\Mm}{{\mathcal M}}
\newcommand{\Ss}{{\mathcal S}}

\newcommand{\ov}{\overline}

\renewcommand{\Hat}{\widehat}
\renewcommand{\Tilde}{\widetilde}

\newcommand{\Aa}{{\mathcal A}}
\newcommand{\Xx}{{\mathcal X}}

\newcommand{\Q}{{\mathbb Q}}
\newcommand{\R}{{\mathbb R}}
\newcommand{\C}{{\mathbb C}}

\newcommand{\Aut}{{\rm Aut}}

\newcommand{\Z}{{\mathbb Z}}

\newcommand{\Symp}{{\rm Symp}}

\newcommand{\Diff}{{\rm Diff}}

\newcommand{\Nn}{{\mathcal N}}
\newcommand{\Jj}{{\mathcal J}}

\newcommand{\Gg}{{\mathcal G}}

\newcommand{\ev}{{\rm ev}}
\newcommand{\bp}{{\bf p}}
\newcommand{\HE}{{\Hat E}}
\newcommand{\HL}{{\Hat L}}
\newcommand{\SSS}{{\smallskip}}
\newcommand{\QED}{{\hfill $\Box$\MS}}

\newtheorem{theorem}{Theorem}[section]
\newtheorem{thm}[theorem]{Theorem}

\newtheorem{cor}[theorem]{Corollary}
\newtheorem{lemma}[theorem]{Lemma}

\newtheorem{prop}[theorem]{Proposition}

\newtheorem{defn}[theorem]{Definition}
\newtheorem{example}[theorem]{Example}

\newtheorem{rmk}[theorem]{Remark}

\numberwithin{figure}{section}
\numberwithin{equation}{section}
\numberwithin{table}{section}

\newcommand{\MS}{{\medskip}}

\newcommand{\NI}{{\noindent}}

\begin{document}
 \title{Some $6$ dimensional Hamiltonian $S^1$-manifolds}
 \author{Dusa McDuff}\thanks{partially supported by NSF grant DMS 0604769.}
\address{Department of Mathematics,
Barnard College, Columbia University, New York, NY 10027-6598, USA.}
\email{dmcduff@barnard.edu}
\date{August 26, 2008, revised May 13 2009}
\begin{abstract}
In an earlier paper we explained how to convert the problem of symplectically embedding one $4$-dimensional ellipsoid into another into the problem of embedding a certain set of 
disjoint balls into $\C P^2$ by using a new way to desingularize 
orbifold blow ups $Z$ of the weighted projective 
space $\C P^2_{1,m,n}$.
We now use a related method to construct symplectomorphisms
of these spaces $Z$.   
This allows us to construct some 
well known Fano $3$-folds (including the Mukai--Umemura $3$-fold)
in purely symplectic terms using a classification by Tolman of 
a particular class of Hamiltonian $S^1$-manifolds.
 We also show that (modulo scaling) these manifolds are uniquely determined by their fixed point data up to equivariant symplectomorphism. As part of this argument 
 we show that  the symplectomorphism group of a certain
  weighted blow up of a weighted projective plane is connected.
\end{abstract}

\maketitle

\tableofcontents

\section{Introduction}

\subsection{Statement of results}

In \cite{T}, Tolman considers the problem of
classifying all symplectic $6$-manifolds $(M,\Om)$ with a
Hamiltonian $S^1$ action in the case when $H^2(M;\Z)$ has rank $1$.
She proved that under these assumptions $H^*(M;\Z)$ is additively
 isomorphic to $H^*(\C P^3;\Z)=\Z$ and that there are 
 four possibilities for the number  $\ell: = 6-c_1(\be)$ where $\be$ is the  generator of $H_2(M)$ with $\om(\be)>0$. The two standard cases are $M_2=\C P^3$  and $M_3 = \Tilde{G}_\R(2,5)$, the Grassmannian of oriented $2$-planes in $\R^5$ 
(known to complex geometers as the quadric surface in $\C P^4$).

   However there are two other possibilities,
with $\ell = 4$ or $5$.  In the latter two cases Tolman showed that
$S^1$ must act with precisely $4$ fixed points $x_k, 1\le k\le 4,$ that have index $8-2k$ and isotropy weights $w_k$, where
$$
w_1=(-1,-2,-3), \quad w_2=(1,-1,-\ell),\quad w_3=(1,\ell,-1),\quad
w_4=(1,2,3).
$$
Moreover, the generating Hamiltonian $H$ (or moment map) can be chosen\footnote
{There are two choices here.  The first is to choose the additive constant for $H$ so that it is symmetric about $0$, and the second is to scale the symplectic form so that it equals $c_1(M)$.}
 to have critical levels
$$
H(x_1) = 6,\quad H(x_2) = \ell,\quad H(x_3) = -\ell,\quad H(x_4) = -6,
$$
and the integral  cohomology ring $H^*(M_\ell;\Z)$ must have the following form:
if $x\in H^2$ and $ y\in H^4$ are generators such that $x(\be)=1$ and $x y$ generates $H^6$, then
$$
x^2=5y \mbox{  when }\ell=4,\;\;\mbox{and}  \;x^2=22y
\mbox{  when }\ell=5.
$$
Tolman showed that this data satisfies many consistency checks.
However she left open the question as to whether  manifolds 
with $\ell=4,5$ actually exist.   

It turns out that these manifolds are well known to complex geometers.
Any Hamiltonian $S^1$ manifold contains $2$-spheres on which $\om$ is positive; take the $S^1$-orbit of any $g_J$-gradient flow line of 
the moment map $H$, where $g_J: =\om(\cdot,J\cdot)$ is defined using a compatible almost complex structure $J$.  Hence, if complex, these manifolds  
would be Fano $3$-folds with $b_2=1$ and $b_3=0$.
Such manifolds are classified (see \cite[Ch~12]{IP}).  There are precisely four families, corresponding to the four cases $\ell=2,3,4,5$ discussed above.  
Rather than using the number $\ell$,
algebraic geometers  
distinguish them by their index  $r: = c_1(\be) = 6-\ell$.
When $r=4$ one has    $\C P^3$ and when $r=3$ the quadric.   There
is a unique  complex manifold $V_5$ (also sometimes called $B_5$)
 with index $2$ which is rigid (i.e. its complex structure does not deform); it supports a nontrivial action of 
$SL(2,\C)$.  In contrast, when $r=1$ there is a family $V_{22}$  of   manifolds.  As shown by Prokhorov~\cite{Pro},  
there is a unique member of this family  $V^s_{22}$ with
a nontrivial $SL(2,\C)$ action, another unique member $V_{22}^a$ with an action of $\C$ and a family $V_{22}^m$ depending on one rational parameter with an action of $\C^*$.
The manifold $V^s_{22}$ was first constructed by 
Mukai--Umemura \cite{MU}
and is of particular interest to geometers because of its 
K\"ahler--Einstein metrics; cf. Donaldson \cite{D} for example.

In this paper we construct
 the manifolds $M_4=V_5$ and $M_5=V_{22}$ in purely symplectic terms.
 We also show that they admit complex structures that are invariant 
 under an $S^1$ action and hence under a $\C^*$ action.
 Because they are Fano, they also have 
$S^1$-invariant
K\"ahler structures induced by the embeddings into projective space
provided by sections of high enough powers of the anticanonical bundle.
Our method does not exhibit the $SO(3)$ action (but see Remark~\ref{rmk:so3} and  \cite[\S5.2]{D}).

\begin{thm}\labell{thm:main} {\rm (i)} When $\ell = 4,5$, there are
Hamiltonian $S^1$ manifolds $(M_\ell,\Om)$ 
with the properties described above.  Modulo scaling, they are unique up to $S^1$-equivariant symplectomorphism.\SSS

\NI {\rm (ii)} Moreover these manifolds may be given
an $S^1$-invariant complex structure. This is unique when $\ell=4$, and depends on a rational parameter when $\ell=5$.
\end{thm}

The only new statement above is the uniqueness part of (i).
Its proof takes the approach proposed by Gonzalez \cite{Gonz} and relies on Theorem~\ref{thm:rigid} which states that the reduced spaces are \lq\lq rigid", i.e. that their symplectic structures are unique in a fairly strong sense.
The construction of the complex structures in (ii) is rather different from those in the original papers (cf. \cite{MU}), and provides a new
 perspective on the discussion of the $SO(3)$ action in 
 Donaldson \cite[\S5.2]{D}.

We analyze the symplectic structure of  $(M_\ell,\Om)$ via the family of reduced spaces.  As is explained in more detail below, these reduced spaces are $4$-dimensional symplectic orbifolds. Rather than looking at them directly as in Chen \cite{Ch}, we study them via their symplectic resolution as in McDuff \cite{M}. 
The resolution of the middle reduced level is the blow up  
$X_k$ of  $\C P^2$ at $k: = \ell + 3$ points, and our first construction
is based on the existence of certain elements of order two (the Geiser and Bertini involutions)  in the plane Cremona group; cf. Remark~\ref{rmk:Bert}(ii).  Although our method is applied here only in a special case, 
 in principle  it could be used  
to construct  any $6$ dimensional Hamiltonian 
$S^1$-manifold with isolated fixed points once one has a consistent set of fixed point data.  However, the uniqueness result uses the fact the resolution involves a relatively small number of blow ups, and may well not hold in general. Note also that
 the existence of complex structures on $M_\ell$ is established
by a somewhat different argument, one 
relying on the existence of very special complex 
structures that are invariant under analogs of the above involutions: see \S\ref{ss:com}.  
 
\subsection{Sketch of proof.}

We now sketch our argument in the symplectic case.
In \cite{God}, Godinho analysed the change in structure of the reduced spaces of a Hamiltonian $S^1$-manifold when one passes through a critical point of index or coindex $2$.  Her work implies that 
if the manifolds $M_\ell$ exist then the regular reduced spaces $(Z_\ka,\om_\ka)$ at level $\ka\in (-6,6)$ must be certain orbifold blow ups of weighted projective spaces.  Tolman worked out precisely 
what these reduced spaces must look like (see Lemma~\ref{le:T} below), and pointed out that the 
question of whether they actually 
 exist is equivalent to an  ellipsoidal embedding problem.
 The latter problem was solved in \cite{M}.  It follows immediately that  the 
sub- and super-level sets 
$$
\bigl(M_\ell^{\le 0},\Om\bigr): = \bigl(H^{-1}([-6,0]),\Om\bigr),\quad
\bigl(M_\ell^{\ge 0},\Om\bigr):=\bigl(H^{-1}([0,6]),\Om\bigr)
$$ 
of $M_\ell, \ell=4,5,$ also exist. 
Therefore all we need to do  is  glue
 the boundary of $\bigl(M_\ell^{\le 0},\Om\bigr)$ to that of $\bigl(M_\ell^{\ge 0},\Om\bigr)$.  
 
 If $Z_0$ had no singularities, this would amount to constructing the symplectic sum of the cut symplectic manifolds $(M^-,Z^-,\om^-)$ and $(M^+,Z^+,\om^+)$ along the copies
 $Z^-, Z^+$ of $Z_0$, where $M^-$, for example, is obtained from 
 $M^{\le 0}$
 by collapsing each $S^1$ orbit in its boundary to a point in $Z^-$.
 For this sum operation to be possible we need there to be a 
 symplectomorphism $(Z^-,\om^-) \to (Z^+,\om^+)$ that reverses the sign of  the Euler class of the normal bundles.    In the case at hand,
the boundary $(Y^-, \Om^-): = (H^{-1}(0), \Om)\subset M^{\le0}$
is the (smooth)  total space of a principal
  $S^1$-orbibundle $\pi: (Y^-,\Om^-)\to (Z^-,\om_0^-)$ 
  over the reduced space 
 $(Z^-,\om_0^-): = (Y^-/S^1,\om^-)$,
 which is a symplectic orbifold whose singular set $\bp$ consists of $3$ 
 points.   It is not hard to see that the orbibundle $Y^-\to Z^-$ is determined by its restriction to $Z^-\less\bp$.  Since the latter is a circle bundle,
 it is in turn
  determined by its Euler class  $e(Y^-)\in H^2(Z^-\less\bp;\Z)$.   
 But, as we shall see in 
 \S\ref{ss:wp}, $H^2(Z^-\less \bp; \Z)$ is a free abelian group, and the restriction map $H^2(Z^-;\Q)\to H^2(Z^-\less \bp; \Q)$ is an isomorphism.  
 Hence the orbibundle $Y^-\to Z^-$ is determined by the unique class 
 $e_Z(Y^-)\in H^2(Z^-;\Q)$ that restricts to $e(Y^-)$. 
 This leads to the following statement.

 \begin{lemma} \labell{le:eul} To construct $(M_\ell,\Om)$ as a Hamiltonian $S^1$-manifold,  it suffices to find a symplectomorphism  
 $\phi_Z:(Z^-,\om_0^-)\to (Z^+,\om_0^+)$ such that
 $\phi_Z^*(e_Z(Y^+)) = -e_Z(Y^-)$.
 \end{lemma}   
 
 Therefore the first part of the following result gives the existence statement of Theorem~\ref{thm:main} (i), while the second part will imply the uniqueness statement via Lemma~\ref{le:reg}.
  
\begin{prop}\labell{prop:ell} {\rm (i)} For $\ell=4,5$, there is a symplectomorphism $\phi_Z:(Z^-,\om_0^-)\to (Z^+,\om_0^+)$ such that
$\phi_Z^*(e_Z(Y^+)) = -e_Z(Y^-)$. \SSS

\NI {\rm (ii)}  Moreover $\phi_Z$ is unique up to symplectic isotopy.
\end{prop}

To prove this we resolve $Z$ as follows.  
Denote by $(X_k,J_0)$ the complex manifold obtained by  blowing up  $\C P^2$ at $k$ generic points, and by $L, E_i,1\le i\le k,$ 
 the classes of the line $\C P^1$ and the 
$k$  exceptional divisors.  We shall provide $(X_k,J_0)$ with a $J_0$-tame symplectic form  in the class 
\begin{equation}\labell{eq:tau}
[\tau] = 3a-\sum_{i=1}^k e_i = c_1(X_k,J_k),
\end{equation}
where $a, e_i$ are Poincar\'e dual to $L, E_i$ respectively.  In particular,
 $e_i(E_j) =- \de_{ij}$. 
 Note that it does not matter here how we choose $J_0$ or the symplectic form;
 by \cite{Mdef}, any choices give forms that are deformation equivalent and hence isotopic.
 Further, define 
\begin{eqnarray}\labell{eq:chi}
 \chi_7: &=& \ts{\frac 1{12}}\bigl(6a-\sum_{1\le i\le 3} 2e_i - \sum_{4\le i\le 7} 3e_i\Bigr)\quad\mbox{ on } X_7,\\\notag
 \chi_8: &=& \ts{\frac 1{30}} \bigl(15a-\sum_{1\le i\le 3} 5e_i - \sum_{4\le i\le 8} 6e_i\bigr)\quad\mbox{ on } X_8.  \end{eqnarray}

  We shall see in \S\ref{s:2} that there is a complex structure 
  $J (\ne J_0)$ on $X_{\ell+3}$ and a holomorphic blow down map 
  $\Phi_J:(X_{\ell+3},J)\to Z^-$ such that 
  $$
  \chi_{\ell+3} = - \Phi_J^*(e_Z(Y^-)),\quad [\tau]
   = \Phi_J^*([\om_0]).
  $$
    In fact, if one thinks of $Z^-$ as a complex orbifold, $\Phi_J$ is just a standard resolution of 
  its singularities; the results of  
  \cite{M} are needed only to understand the 
  symplectic structure of $Z^-$.   Similarly, there is a 
holomorphic blow down map    $\Phi_{J'}:(X_{\ell+3},J')\to Z^+$
such that
 $\chi_{\ell+3} =  \Phi_{J'}^*(e_Z(Y^+))$.  These facts, together with Proposition~\ref{prop:Zuniq} concerning  the uniqueness of symplectic forms on  $Z^{\pm}$, allow us to reduce the proof of Proposition~\ref{prop:ell} (i) to the following lemma.   
  
 \begin{lemma}\labell{le:2} For $k=7,8$, there is a diffeomorphism
$\psi:X_k\to X_k$  such that $\psi^*(\chi_k) = -\chi_k$.
\end{lemma}

This result is classical (cf. Remark~\ref{rmk:Bert}), but we
 prove it  in \S\ref{ss:1} 
for the sake of completeness.  
 This completes the 
construction of $(M_\ell,\Om)$ as a symplectic manifold.  Here the resolution $X_k$ is for the most part considered as a complex manifold and we use the holomorphic blow down map $\Phi_J:X_k\to Z$.  However, to prove uniqueness we need to understand the symplectic structure of $Z$  much more deeply. In particular the following result is proved in \S\ref{ss:Z}.

\begin{prop}  For any symplectic structure on the orbifold $Z$
the group of symplectomorphisms 
 that act trivially on homology is connected.
 \end{prop} 

The proof
uses the symplectic version of the resolution.  In Lemmas \ref{le:reg} and \ref{le:sing} we also give proofs of basic uniqueness results for
suitable slices $H^{-1}(a,b)$ of Hamiltonian $S^1$-manifolds.  
These lemmas are well known, but there is no convenient reference in the literature.

\begin{rmk}\labell{rmk:inter}\rm    (i)
We explain in \S\ref{ss:23} a similar construction for the manifolds $M_2=\C P^3$ and $M_3=\Tilde G_{\R}(2,5)$.  Since
the $S^1$ action on $M_2$ extends to a Hamiltonian action of $T^3$,
 the reduced spaces in this case 
 are toric, with moment polytopes given by a 
family of parallel slices of the $3$-simplex
that is illustrated in Figure \ref{fig:12}. \MS
 
 \NI (ii)   $M_4$ and $M_5$ admit Hamiltonian $SO(3)$ actions,
 and it would be interesting to use the methods of 
 River Chang \cite{Chang} to understand them up to 
 $SO(3)$-equivariant symplectomorphism.  More generally, it would be interesting to understand when a Hamiltonian $S^1$ action extends to
 an $SO(3)$ action; can one give conditions on the 
 reduced spaces that would guarantee this?  The toric version of 
 this question is understood. For example, it is shown in McDuff--Tolman \cite{MT2}  that a 
 toric manifold admits a compatible 
 $SO(3)$ action if and only if the moment polytope admits a nontrivial 
 robust affine symmetry; cf. \cite{MT2} Lemma~1.26 and Proposition~5.5.
\end{rmk}

\NI {\bf Acknowledgements.\,}  I am very grateful to Susan Tolman for showing me an early version of her paper \cite{T},  to Weimin Chen and Eduardo Gonzalez for some helpful comments on a previous version of this note, and to the anonymous referee for many small helpful suggestions. 
 Also I owe a debt of gratitude to the many people who helped me with various aspects of algebraic geometry, in particular Alessio Corti,  Ragni Piene, Paul Seidel,  Jason Starr, and Balazs Szendroi.
  Any remaining mistakes are of course the responsibility of the author.

\section{Blow ups of $\C P^2$ and weighted projective spaces.}\labell{s:2}

\subsection{Symplectomorphisms of $X_k$}\labell{ss:1}

In this section we shall prove  Lemma~\ref{le:2} in the 
more precise form given by Proposition~\ref{prop:psi} below.  
We begin with a general discussion of automorphisms of $X_k$.
 One difficulty in making this
 discussion precise is that there are serious gaps in our knowledge of the group $\Diff(X_k)$ of diffeomorphisms of $X_k$.  In particular,
even when $k=0$, i.e. for $X_0 = \C P^2$,
it is not known whether  the subgroup  $\Diff_H(X_k)$
that acts trivially on homology is connected, though the group of symplectomorphisms of $\C P^2$ is connected by Gromov's results.

For all $k$ we shall denote by $J_0$ the complex structure on $X_k$ 
obtained by identifying $X_k$  with the blow up of $\C P^2$ at a particular set of $k$ generic points.  We shall assume
that  $(X_\ell, J_0)$ is a blow up of $(X_k, J_0)$ for all $\ell>k$ and write $K$ for its canonical class.  Thus 
$-K = 3L-\sum_{i=1}^k E_i$.

We shall denote by 
$\Ee(X_k)$ the set of classes in $H_2(X_k)$  that can be 
represented by embedded $-1$ spheres. Thus 
$\Ee(X_k) = \{E\in H_2(X_k): E^2=-1, K\cdot E = -1\}$. 
When $k\le 8$ the elements of $\Ee(X_k)$ can be listed as follows
(modulo permutations of the indices)
\begin{eqnarray}\labell{eq:Ek}
&&E_1,\;\;L-E_{12};\;\; 2L-E_{1\dots5}; \;\;3L-2E_1-E_{2\dots 7};\\\notag
&&4L - 2E_{123}-E_{4\dots 8};\;\;
5L-2E_{1\dots 6}-E_{78};\;\;
6L-3E_1-2E_{2\dots 8}.
\end{eqnarray}
(Here we denote $\sum_{i=j}^n E_i =:E_{j\dots n}$. Further, elements of the last three kinds do not appear in $\Ee(X_7)$ since they involve $8$ different $E_i$.)

Next, recall that the classical {\it Cremona transformation} 
 $R_{123}:X_3\to X_3$ is the biholomorphism that 
 covers the birational map
$$
\rho:\C P^2\less \{3\mbox{ pts}\} \;\longrightarrow\;
\C P^2\less \{3\mbox{ pts}\},\qquad   [x:y:z]\mapsto [yz:xz:xy].
$$
Thus $R_{123}$ acts on $H_2(X_3)$ by
$$
L\mapsto 2L-E_{123},\qquad E_i\mapsto L-E_j - E_k,
$$
where $\{i,j,k\}=\{ 1,2,3\}$.
By Seidel \cite{Sei}, $R_{123}$  is isotopic to
 a symplectomorphism of $X_3$, when this has 
 a $J_0$-tame symplectic form in the class 
 Poincar\'e dual to  $-K= 3L- E_{123}$. Indeed,  in this case $R_{123}$
 is isotopic to the Dehn twist 
 in a Lagrangian sphere in class 
 $L-E_{123}$.  
 
 We denote the Cremona transformation of $X_k$  in the exceptional divisors $E_i, E_j, E_\ell$ by $R_{ij\ell}$.  It is well defined up to isotopy, and acts on
$H_2(X_k)$ by the reflection
$A\mapsto A+(A\cdot B) B$ where $B : = L-E_{ij\ell}$.

Denote by
$\Aut_K(X_k)$ the group of automorphisms of the homology group $H_2(X_k;\Z)$ that preserve the canonical class $K$ 
and the intersection form.  Further, denote by 
$\Diff_K(X_k)$ the group of diffeomorphisms of $X_k$ that preserve
$K$.  A classical result of Wall \cite{W} asserts that the natural map 
$\pi_0(\Diff_K(X_k)) \to \Aut_K(X_k)$ is surjective when $k\le 9$.
Moreover, its image is generated by permutations of the $E_i$ and the Cremona transformations $R_{ij\ell}$.\footnote
{
When $k\le 8$ this is easy to verify directly since $\Ee(X_k)$ 
is finite with elements as listed in (\ref{eq:Ek}).  For example, the following composite takes $E_1$ to $\Hat E_1: = 3L-2E_1-E_{234567}$:
$$
E_1\;\;\stackrel{R_{123}}\longrightarrow \;\;L-E_{23}
\;\;\stackrel{R_{145}}\longrightarrow \;\;2L- E_{12345}
\;\;\stackrel{R_{167}}\longrightarrow \;\;3L-2E_1-E_{234567}.
$$}
   
 Consider the following elements of $H_2(X_7;\Z)$:
\begin{eqnarray}\labell{eq:Hat}
\eps_7: &=& \ts{\frac 1{12}}\bigl(6L-2E_{123} - 3E_{4567}\bigr),\\\notag
\HL: &=& 8L-3E_{1\dots7},\\\notag
\HE_i: &=& 3L-2E_i -\sum_{j \ne i} E_j,\quad i=1,\dots,7;
\end{eqnarray}
and of $H_2(X_8;\Z)$:
\begin{eqnarray}\labell{eq:Til}
\eps_8: &=& \ts{\frac 1{30}}\bigl(15L-5E_{123} - 6E_{45678}\bigr),\\\notag
\TL: &=& 17L-6E_{1\dots 8}\\\notag
\TE_i: &=& 6L-3E_i -2\sum_{j \ne i} E_j \quad i=1,\dots,8.
\end{eqnarray}

\begin{prop}\labell{prop:psi} For $k=7,8$,  there is a diffeomorphism $\psi:X_k\to X_k$ in $\Diff_K(X_k)$ that takes the classes
$L, E_i$ to $\HL, \HE_i$ when $k=7$ and to $\TL, \TE_i$ when $k=8$.
 Moreover $\psi_*(\eps_k) = -\eps_k$. 
\end{prop}
\begin{proof}   
By the  results of Wall  mentioned above, it suffices to prove that there is 
an element of $\Aut_K(X_k)$  with this action.  But $H_2(X_k)$ is generated by the classes $L,E_i$ with relations 
$$
L^2=1=-E_i^2,\quad
L\cdot E_i = E_i\cdot E_j=0\;\;\mbox{ if } i\ne j.
$$
Further $K$ is determined by the identities $K\cdot L = -3, K\cdot E_i = -1$.
Therefore to prove the first statement in the case $k=7$, one simply needs to check that the following identities hold for all $1\le i,j\le 7$:
$$
\Hat L^2 = 1,\;\;  \Hat E_i\cdot \Hat E_j = -\de_{ij}, \;\;\Hat L\cdot\Hat E_i = 0,\;\;
K\cdot\Hat E_i = -1,\;\; K\cdot\Hat L = -3.
$$ 
A similar argument works when $k=8$.

The last statement holds because
\begin{equation}\labell{eq:eps}
-\eps_7= \ts{\frac 1{12}}\bigl(6\HL-2\HE_{123} - 3\HE_{4567}\bigr),\quad
-\eps_8= \ts{\frac 1{30}}\bigl(15\TL-5\TE_{123} - 6\TE_{45678}\bigr).
\end{equation}
This completes the proof.
\end{proof}

\begin{rmk}\labell{rmk:Bert}\rm  {\rm (i)}
As we shall see in \S\ref{ss:com}, there are other possibilities for $\psi$.
However, they all involve classes of the type  
 $\HE_i$ and $\TE_j$.  As is shown in
  the proof of Proposition 1.5 in \cite{M},
  these are precisely the 
 classes that give the obstructions to embedding
$\la E(1,\ell)$ into $E(2,3)$ for large $\la$.
Hence their size must decrease to $0$ as one approaches the critical value $\ka=\ell$ from below, so that they are natural candidates
for the classes of the exceptional divisors created 
as $\ka$ decreases through $\ell$. 
\SSS

\NI
{\rm (ii)}  
For sufficiently generic complex structures on $X_k$  one can choose the map $\psi$ to be a biholomorphic involution.  When $k=7$ one gets 
the family of  Geiser involutions, while when $k=8$ one gets the Bertini involutions.  They may be recognized by the fact that
in each case the sum $A + \psi_*(A)$ for $A\in H_2(X_k)$ is always a multiple of the canonical class $K=-3L+\sum E_i$;
cf.  Dolgachev--Iskovskikh \cite{Dolg}.   No doubt one could use this fact to construct complex structures on $M_\ell$.  But
because we are interested in the singular complex structures on $X_k$ that are pulled back from $Z$, one would need to look at the moduli spaces of these involutions quite carefully.  In \S\ref{ss:com} we shall take a somewhat different approach.
\end{rmk}

\subsection{Resolving weighted projective spaces.}\labell{ss:wp}

We first describe the  reduced manifolds $(Z,\om)$.
Since these are weighted blow ups of weighted projective spaces, we shall begin with some background information on these spaces.
For further details, see Godinho \cite{God}.

Let $\un m:= (m_1,\dots,m_N)$ where the $m_i$ are positive integers.
Denote $a_i: = \prod_{j\ne i}m_j$ and $A: = \prod m_i$, so that
$a_im_i = A$ for all $i$.
By definition, the weighted projective space 
 $
W: =  \C P^{N-1}_{\un m} $
  is the complex orbifold obtained 
   by quotienting  $\C^N\less\{0\}$ 
  by the group $\C^*$ acting via
 $$
\la\cdot  (z_1,\dots,z_N) = (\la^{m_1}z_1,\dots,\la^{m_N}z_N).
$$
We shall normalize the symplectic form $\om_0$ on $\C^N$ so that the Hamiltonian function for the induced Hamiltonian action 
of $S^1\subset \C^*$ on $(\C^N,\om_0)$ 
is 
$$
H_{\un m}: = \sum m_i|z_i|^2 = A(\sum \frac{|z_i|^2}{a_i}).
$$
Then    $\C P^{N-1}_{\un m}$ may also be considered as one of the reduced spaces of this action and given the corresponding symplectic form
$\tau_{\un m}$.  To keep our coefficients integral, 
we shall identify it with the reduced space at  level $A$.  Thus 
$$
(\C P^{N-1}_{\un m}, \tau_{\un m}) = H^{-1}(A)/S^1
$$
 is the  quotient of
   the boundary of the ellipsoid
 $$
 E(\un a): = \left\{z\,|\,\sum \frac{|z_i|^2}{a_i}\le 1\right\}\subset \C^N
  $$ 
 by the characteristic flow.   
  Note that, for any $c>0$, the rescaled space
 $
 (\C P^{N-1}_{\un m}, c\tau_{\un m})$ is the similar quotient 
 of the boundary $H^{-1}(cA)$ of 
\begin{equation}\labell{eq:wp}
 c\,E(\un a): = \left\{z\,|\,\sum \frac{|z_i|^2}{a_i}\le c\right\}.
\end{equation}

By construction, the weighted projective space $W$ is a toric manifold whose
 moment polytope $\De_W$ can be identified with the intersection of the hyperplane
 $\sum \frac{x_i}{a_i} = 1$ with the positive quadrant $\{x_i\ge 0\}$ in $\R^N.$
  If $m_1=1$ then the vertex $(1,0,\dots,0)$ of $\De_W$ is smooth,
 and there is an integral affine transformation of $\R^N$ that takes this vertex to $0$ and takes
 $\De_W$  to the polytope
 $$
 \De_{a_2,\dots,a_N}: = \{\un x\in \R^{N-1}\,|\,
 x_2,\dots,x_N\ge 0, \; \sum_{i>1}\frac {x_i}{a_i}\le 1\}.
 $$
 Therefore, in this case we can think of $W$ as the compactification of the interior of the ellipsoid $E(a_2,\dots,a_N)$ that is obtained by adding the quotient of the boundary in which
  each orbit of  the characteristic flow is collapsed to a point.

\begin{example}\labell{ex:ex}\rm  Let us
specialize to the case $N=3$.  If $\un m = (1,p,q)$, then
$a_1=A=pq, a_2=q$ and $a_3=p$.  Therefore
 the moment polytope  $\De_W$ of $W: = \bigl(\C P^2_{1,p,q}, \tau_{1,p,q}\bigr)$ is the triangle $T_{q,p}$
in $\R^2$ with vertices $(0,0), (q,0)$ and $(0,p)$; see Figure \ref{fig:12} and \cite{M}.\footnote
{
For a general treatment of toric symplectic orbifolds see Lerman and Tolman \cite{LT}.}   
(The fact that the weights $q,p$ of the ellipsoid  coincide  modulo order with 
the initial weights $m_i, i>1,$ is an accident that happens in this dimension only.) 

As always, this moment polytope determines the 
symplectic form $\tau_{p,q}$: indeed,
for every edge $\ve$ of the moment polytope $\De_W$,
the integral of $\tau_{p,q}$ over $\ve$ equals the {\it affine length} of 
$\ve$.  This can be measured as follows.  Take any  affine transformation $A$ of $\R^2$
that preserves the integer lattice and is such that $A(\ve)$ lies along the $x$-axis, and then measure the Euclidean length of $A(\ve)$. Thus if $\ve$ has rational slope and  endpoints on the integer lattice, $\al(\ve) =k+1$ where $k$ is the number of  points 
of the integer lattice in the interior of $\ve$.  In particular, if $p,q$ are mutually prime,
 \begin{equation}\labell{eq:int}
\int_{\C P^1_{p,q}}\;\tau_{p,q}=1.
 \end{equation}
\end{example}

 The following lemma is due to Tolman \cite{T}. We explain its proof for the convenience of the reader. Note that she uses the form $\om_{1,m,n}: = \frac 1{mn}\tau_{1,m,n}$ on $\C P^2_{1,m,n}$.
  
 \begin{lemma}[Tolman]\labell{le:T} Suppose that the manifold $M_\ell$ exists for some integer $\ell\in [2,5]$. 
Then the reduced space  $(Z,\om_\ka)$ at level $\ka\in (-\ell,\ell)$ 
 is diffeomorphic to the connected sum 
 $\C P^2_{1,2,3}\#\ov{\C P}\,\!^2_{1,1,\ell}$ of the weighted projective space
 $\C P^2_{1,2,3}$ with a conjugate $\ocp^2_{1,1,\ell}$. 
 Moreover the symplectic form $\om_\ka$ lies
in the unique class $[\om_\ka]$ such that
 \begin{equation}\labell{eq:omZ}
 [\om_\ka]|_{\C P^1_{2,3}} ={\ts \frac{6+\ka}6}\,\tau_{2,3}, \quad 
 [\om_\ka]|_{{\C P}\,\!^1_{1,\ell}} = {\ts \frac{\ell+\ka}{\ell}}\,\tau_{1,\ell}.
 \end{equation}

 \end{lemma}
 \begin{proof} It follows from equation (\ref{eq:wp}) that
 the reduced space for the Hamiltonian
 $H: = \sum_{i=1}^3 	m_i |z_i|^2$ at level $\eps>0$ is
  $
 \Bigl(\C P^{2}_{\un m},\frac{\eps}{A} \,\tau_{\un m}\Bigr).
 $
Thus, if $\un m = (1,2,3)$ the reduced space is
$$
 \Bigl(\C P^2_{1,2,3}, \ts{\frac \eps 6} \,\tau_{1,2,3}\Bigr).
 $$
Since the minimal critical level is at $\ka = -6$ rather than $0$,
 the coefficient of $\tau_{2,3}$
  in equation (\ref{eq:omZ}) is therefore $\frac{6+\ka}6$.
 
 To understand the diffeomorphism type of the reduced space
 at level $\ka\in (-\ell,\ell)$, first recall
 from  Example~\ref{ex:ex} that
 when  $\un m= (1,m_2,m_3) = :(1,\un m')$, one can also obtain 
 $(\C P^{2}_{\un m}, \tau_{\un m})$
 from the ellipsoid $E: =E(\un m')\subset \C ^{2}$
 by collapsing its boundary $\p E$ to $\C P^{1}_{\un m'}$ as above. It follows that
the connected sum 
 $X\#\ov{\C P}\,\!^{2}_{1,\un m'}$ can be considered as a  orbifold blow up, in which one cuts out an embedded ellipsoid 
 $\eps E(\un m')\subset X$  for some small $\eps>0$ and then collapses the boundary along the characteristic flow.   This is called the (symplectic) orbifold blow up with weights $\un m'$.  Using toric models one can show 
 that as one passes a critical point with isotropy
weights
$(-1, m_2,m_3)$ (where $m_i>0$) the critical level undergoes an
orbifold blow up with weights $\un m'$.  This is illustrated in
Figure \ref{fig:7} below, and
a detailed proof is given by Godinho~\cite{God}.

 For example, the reduced space at level $\eps>0$ of the function 
 $H = -|z_1|^2 + m_2|z_2|^2 + m_3|z_3|^2$
has as exceptional divisor the quotient of the level set
$$
H(z_2,z_3) = m_2m_3\Bigl(\frac{|z_2|^2}{m_3} + \frac{|z_3|^2}{m_2}\Bigr)=\eps,
$$
which is $\bigl(\C P^1_{m_3,m_2},\frac\eps{m_2m_3} \tau_{m_3,m_2}\bigr)$.
In particular, when $(m_2,m_3) = (1,\ell)$ and the critical point occurs at level $-\ell$, one obtains the coefficient
$(\ell+\eps)/ \ell$ of (\ref{eq:omZ}). 
\end{proof}

\begin{rmk}\labell{rmk:toric}\rm  When $\ka+\ell>0$ is sufficiently small the weighted blow up can be done equivariantly
so that $(Z,\om_\ka)$ has a global toric structure as in Figure \ref{fig:2}.  We shall denote by $J_T$ the corresponding complex structure on $Z$.
\end{rmk}

Observe that $Z$ has three singular points $p_m, m=2,3,\ell$, each with a neighborhood $\Nn_m$ of the form $\Tilde\Nn_m/\Z_m$, where $\Tilde\Nn_m: = B\subset \C^2$ is a (closed) ball with suitable small radius and the generator of $\Z_m: = \Z/m\Z$ acts via
$(z_1,z_2)\mapsto (e^{2\pi i/m}z_1,e^{-2\pi i/m}z_2)$.\footnote
{
These are known in the literature as simple singularities of type $A_{m-1}$: see for example Ohta--Ono \cite{OO}. They may be resolved by chains of $-2$spheres of length $m-1$; cf. Lemma~\ref{le:resol}.}
For example, $\Nn_3$ is a neighborhood of $[0:0:1]$ in 
$\C P^2_{1,2,3}$ and $\la \in \Z_3$ acts by
$$
[z_0:z_1:1]\;\;\mapsto \;\;  [\la z_0:\la^2z_1:\la^3] = [\la z_0: \la^{-1}z_1:1].
$$  
 We shall denote $\bp: = \{p_2,p_3,p_\ell\}$ and
  $\Nn: = \cup_m \Nn_m$.
   By the equivariant Darboux theorem we may (and will) suppose that any symplectic form $\om$ on $Z$ lifts to the standard form $\Tilde\om_0: = \sum_j dx_j\wedge dy_j$ 
  on the local uniformizers $\Tilde\Nn_m$, where $z_j: = x_j + iy_j$.
  
  Although $Z$ can be given an orbifold structure, it is better to think of it as a manifold with singular points.  Since  the order of these singularities are different, any diffeomorphism of $Z$ must fix each $p_m$.  Then the condition for $\phi:Z\to Z$ to be a diffeomorphism is that
  its restriction to the manifold $Z\less\bp$ is smooth and  that for each $m$ there is an open, $\Z_m$-invariant neighborhood $\Tilde U_m$ of $0$ in $\Tilde \Nn_m$ and a diffeomorphism $\Tilde\phi_m: (\Tilde U_m, \Tilde\om_0)\to (\Tilde \Nn_m,\Tilde\om_0)$ that takes $\Z_m$-orbits to  $\Z_m$-orbits; i.e. the following diagram commutes
 \begin{equation}\labell{eq:locphi}
  \begin{array}{ccc}
  \Tilde U_m&\stackrel{\Tilde\phi_m}\to& \Tilde \Nn_m\\
  \downarrow&&\downarrow\\
  U_m&\stackrel{\phi}\to&\Nn_m.
\end{array}
\end{equation}
  Standard arguments show that any  
 diffeomorphism can be isotoped to one
  that is linear with  respect to these local coordinates near $\bp$.   Hence we shall assume that the $\Tilde\phi_m$ are linear.  It is then clear that
  for each $m$ there is an automorphism $\al_m:\Z_m\to \Z_m$ such that
 \begin{equation}\labell{eq:locact} \Tilde\phi_m\circ \ga = \al_m(\ga)\circ\Tilde\phi_m,\quad \ga\in \Z_m.
 \end{equation}
Similarly,  a diffeomorphism $\phi:(Z,\om)\to (Z,\om')$
  is called a {\it symplectomorphism}  if its restriction to the manifold  
  $(Z\less \bp, \om)$ is a  symplectomorphism, and if the local lifts 
  $\Tilde\phi_m$ preserve $\om_0$. 
  
  Note finally that because we are thinking of $Z$ as a singular space, rather than as an orbifold, we define its homology and cohomology groups to be those of the underlying topological space.

 \begin{lemma}\labell{le:phiZ} {\rm (i)} Every diffeomorphism 
  $(Z,\om)\to (Z,\om')$ is isotopic to a 
 diffeomorphism  $\phi$ such that each local linear model $\Tilde\phi_m$
 is  either the identity map or, when $m=3,\ell$, has
  the form $(z_1,z_2)\mapsto (z_2,z_1)$.\SSS
  
  \NI {\rm (ii)} Denote by $e(Y)\in H^2(Z\less\bp;\Z)$ the Euler class of an
   $S^1$-bundle $Y|_{Z\less \bp}\to Z$.  If $\psi^*(e(Y))=-e(Y)$,
  then $\phi$ has the local model $(z_1,z_2)\mapsto (z_2,z_1)$ for $m=3,\ell$, but if $\phi^*(e(Y))=e(Y)$ then $\phi$ is locally modelled by the identity map.\SSS
  
  \NI {\rm (iii)}  The above statements hold also for symplectomorphisms.
 \end{lemma}
 \begin{proof}
 (i) follows from the above discussion  because  there is only one nontrivial equivariant automorphism $\al_m$ of $\Z_m$ when $m=3,\ell$, namely $\ga\mapsto \ga^{-1}$, while there are none for $m=2$.  (ii) holds because 
$\phi^*(e(Y))=-e(Y)$ only if $\phi$ induces the nontrivial automorphism on $\Z_m$ for $m=3,\ell$, while $\phi^*(e(Y))=e(Y)$ only if 
the induced automorphisms on $\Z_m$ are trivial. The proof of (iii) is similar.
\end{proof}

To go further, we need to consider the relation between $Z$ and its resolution $X_k$, where $k=\ell+3$.
We construct
the complex manifold $(X_k,J)$ 
 from $\C P^2$ by blowing up $k$ times (in the complex category)\footnote
 {
 In this paper, there is constant interplay between complex and symplectic blowing up;  the former procedure replaces a point by the family of complex lines through that point, while in the latter
replaces a ball or ellipsoid by the curve obtained by collapsing its boundary.}  as follows. 
 
   Roughly speaking $X_k$ is obtained by blowing $\C P^2$ up three times at one point $p$ and 
 $\ell$ times at another point $q$.   However there are several 
 inequivalent ways of doing this.  By blowing up repeatedly at 
 some point $p$ we mean the following: blow up at $p=:p_1$
 creating an exceptional divisor $C_{E_1}$ in class $E_1$, then blow up at some point $p_2\in C_{E_1}$ obtaining a new exceptional divisor
 $C_{E_2}$ in class $E_2$ and the proper transform $C_{E_1- E_2}$ of $C_{E_1}$, and continue, at the $i$th stage blowing up at some point $p_i$ on the exceptional divisor $C_{E_{i-1}}$
 to obtain $C_{E_i}$ and $C_{E_{i-1}-E_i}$.   We shall only consider the case when  $p_{i+1}\notin C_{E_{i-1}-E_i}$ so that 
the blowing up process  results in a chain of intersecting 
$-2$ curves in the classes $E_1-E_2, E_2-E_3,\dots$.  
Even so, this process is not unique:
although there
 is only one way of doing this twice,  there is a choice at the third blow up.  To see this, suppose that $C_L$ is the unique line in $\C P^2$ through $p_1$ and with proper transform $C_{L-E_1}$ through $p_2$.  Then its proper transform after the second blow up  is $C_{L-E_1-E_2}$, which intersects $C_{E_2}$ at one point.  If we choose $p_3$ to be
  this point of intersection, the third blow up contains curves $C_1, C_2, C_3, C_0$ in classes 
  $E_1-E_2, E_2-E_3$, $E_3$ and $L-E_{123}$, respectively.  In this case we shall say that the blow up at $p$ is directed by $C_L$:  all such blow ups are locally biholomorphic since they depend only on $p$ and $C_L$.  More generally, if $Q$ is an embedded (perhaps noncompact)   holomorphic curve through $p$, we say that repeated blow ups at $p$ are {\it directed by} $Q$ if 
  we always choose the blow up point $p_i\in C_{E_{i-1}}$ to lie on 
  the proper transform of $Q$.

 When constructing the resolution $(X_k,J)$ as a blow up, we 
 always assume that the $3$-fold blow up at $p$ is directed by a line
 $C_L$, and that the $\ell$-fold blow up at $q$ is generic with respect to $p,C_L$. 
   In other words, we assume $q\notin C_L$, and also choose the center  $q_2$ of the second blow up   not on the proper transform
  $C'$  of the line  through $p,q$  so that  $C'$ (which lies in class $L-E_{14}$) lifts to $(X_k,J)$.   For the moment we make no further restrictions
  on the blow up at $q$ (though we will do this in 
 \S\ref{ss:com}).  Therefore, 
  besides the curves $C_0,\dots,C_3$ mentioned above, $(X_k,J)$ contains 
  holomorphic curves $C_4,\dots, C_{k-1}, C_k$ in classes $E_4-E_5, \dots, 
  E_{k-1}-E_k, E_k$ respectively.  We denote by $\Cc$ the set of curves $C_i, 0\le i\le k-1, i\ne 3$:  cf. Fig \ref{fig:4}.  (The curve $C'$ is irrelevant for now, but appears in the proof of Lemma~\ref{le:toric}.)
  
   \begin{figure}[htbp]
    \centering
   \includegraphics[width=2in]{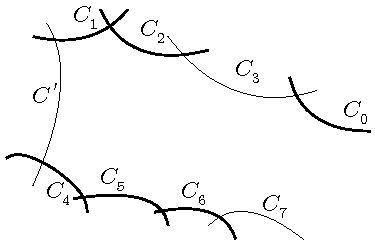} 
   \caption{The curves in $\Cc$ for the case $k=7$,
    together with $C_3, C_7$ and the curve $C'$ in class $L-E_{14}$.}
   \labell{fig:4}
\end{figure}

  Note that  all the curves in $\Cc$  have self-intersection $-2$, and 
 belong to one of three connected components, $C_0, C_1\cup C_2,$ and $C_4\cup\dots\cup C_{k-1}$.  
 It is well known that a string of $-2$ curves of length $s$ blows down to a simple singularity of order $s+1$ and type $A_{s}$.  Thus $C_0$ gives a point of order $2$, 
  $C_1\cup C_2$ a point of order $3$ and 
  $C_4\cup\dots\cup C_{k-1}$ a point of order $\ell$.  
  Hence the blow down of $(X_k,J)$ that contracts these curves gives an orbifold with the same singularities as $Z$.

\begin{figure}[htbp] 
   \centering
   \includegraphics[width=2.5in]{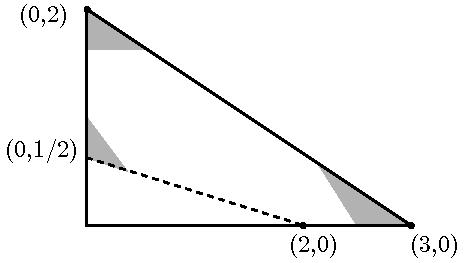} 
   \caption{$\frac 12 T_{1,4}$ embedded in $T_{2,3}$ (where
   $\la T_{a,b}$ denotes the triangle  with vertices $(0,0), (0,\la a),(\la b, 0)$.)  $\De : = \De(\frac 12)$ is defined to be $T_{2,3}\less {\rm int\,}\bigl(\frac 12 T_{1,4}\bigr)$.
   The shaded regions in $\De$ form the image of the neighborhood $\Nn$ of the singular points of $Z$.}
   \labell{fig:2}
\end{figure}

  \begin{lemma}\labell{le:resol} The complex orbifold obtained from $(X_k,J)$ by contracting the three components of $\Cc$ is $Z: = \C P^{2}_{1,2,3}\#{\ov \C P}\,\!^2_{1,1,\ell}$.
  \end{lemma}
  \begin{proof} We shall assume that $\ell=4$ for simplicity.  The case $\ell=5$ is similar.
Denote by $\la T_{a,b}$, where $a<b$, the triangle in $\R^2$ with vertices $(0,0), (0,\la a),(\la b, 0)$; see Figure \ref{fig:2}.  
As explained in Example \ref{ex:ex} the triangle
$T_{2,3}$ is the  the moment polytope for the weighted 
projective space $(\C P^{2}_{1,2,3}, \tau_{1,2,3})$. Similarly, the complement $\De: = \De(\la)$ 
of ${\rm int}\,\bigl(\la T_{1,4}\bigr)$ in $T_{2,3}$ is the moment polytope of its blow up
$(Z, c\om_\ka)$ where $\la: = \frac {3(\ka+4)}{2(\ka+6)}$ and $ c: = 
\frac 6{\ka+6}$.  (These constants can be worked out from equations \ref{eq:int} and \ref{eq:omZ}.)
  
  Fulton explains in \cite{Fu} how to resolve the singularities of
  a toric orbifold by blowing up. Because he is working in the complex 
  rather than symplectic category, he describes the toric variety   by its fan (the set of conormals to the facets); the 
  process of blowing up adds extra elements to the fan.  One can check that the resulting fan is precisely that of the 
  \lq\lq approximation"  $\De_\eps$ to $\De$ that is illustrated in  
 Figure \ref{fig:6} below.  Here the edges going clockwise  from  $C_1^\eps$ have outward conormals:
\begin{gather}\notag
(0,1), \;\; (1,2),\;\; (2,3)^*,\;\; (1,1),\;\; (0,-1)^*,\\ \notag
 (-1,-4)^*,\;\; (-1,-3),\;\; (-1,-2),\;\; (-1,-1),\;\; (-1,0)^*,
\end{gather}
where the starred vectors are also conormals of $\De$.
Therefore $\De_\eps$ is smooth (i.e. the determinant of any successive pair of edges has absolute value $1$), and  it is easy to check that each of its short edges $C_i^\eps$ represents a sphere with self intersection $-2$ (where $i=0,1,2,4\dots,7$). The corresponding symplectic toric manifold is a symplectic version of $(X_7,J)$,
with the curve $C_i$ in $\Cc$ identified to the  short edge $C_i^\eps$.  Thus the resolution described by Fulton is precisely $(X_k,J)$.   
\end{proof}

\begin{defn}\labell{def:res}   For any complex structure $J$ on $X_k$ constructed as above, we shall call the  holomorphic blow down map
$\Phi_J: (X_k,J)\to Z$ the 
{\bf resolution} of $Z$. Further,  we denote by $\Cc$ the collection of curves
$C_i, 0\le i < k, i\ne 3,$  in $X_k$, and by
$$
D_1: = \Phi_J(C_3)\cong \C P^1_{2,3},\quad D_2: = 
\Phi_J(C_k)\cong {\C P}\,\!^1_{1,\ell}
$$
the two divisors in $Z$.
\end{defn}

Note that
$\Phi_J$  is bijective outside the singular points and contracts each connected component of $\Cc$ to one of the singular points of $Z$. 
We shall say more about the resolution as a symplectic manifold later.  For now we shall use it to understand the (co)homology of $Z$.

\begin{lemma}\labell{le:H2} For $\ell = 4,5,$ 
$H^1(Z\less \bp)$ is a free abelian group.  Moreover,
there is a commutative diagram with exact rows
$$
\begin{array}{ccccccccc} 0&\to &H^2(Z;\Z)&\to& H^2(Z\less\bp;\Z)&\to & H^2( \Nn\less \bp;\Z)&\to &0\\
&&\cong\downarrow\;\;\;&&\cong\downarrow\;\;\;&&\cong\downarrow\;\;\;&&\\
 0&\to &\Z^2&\stackrel{\al}\to& \Z^2 &\to &\Z_6\oplus \Z_{\ell}&\to&0,
\end{array}
$$
where the maps in the top row are induced by restriction and where
$\al(m,n) =(6m,\ell n)$.
\end{lemma}

\begin{proof}   We shall prove this for the case $\ell=4$ and then indicate the few changes that need to be made when $\ell=5$.  We  shall
calculate $H^*(Z): =H^*(Z;\Z)$ by 
comparing $Z$ with its resolution $X_7$.  

Denote by $V\subset X: = X_7$ the inverse image $\Phi_J^{-1}(\Nn)$
where $\Phi_J$ is as in Definition~\ref{def:res}.  
Then $\p V \cong \p \Nn$ is a disjoint union of three lens spaces and hence has $H_1(\p V)\cong \Z_2\oplus \Z_3 \oplus \Z_4$, while $H_2(\p V) = 0$.  Thus from the Mayer-Vietoris sequence of the decomposition 
$X = (X\less V) \cup V'$ (where $V'\supset V$ is a slight enlargement of $V$) we obtain the exact sequence
$$
0\to H_2(X\less V) \oplus H_2(V) \to H_2(X)\to H_1(\p V)\to H_1(X\less V)\to 0,
$$
 where we use integral coefficients.  Now $H_2(V)$ is generated by
 $E_1-E_2, E_2-E_3, L-E_{123}, E_4-E_5, E_5-E_6, E_6-E_7$, while
 $H_2(X\less V)$ is generated by those elements of  $H_2(X)$ that are orthogonal to $H_2(V)$ with respect to the intersection pairing. Thus 
 $3L-E_{123}$ and $E_{4567}$ form a generating set for $H_2(X\less V)$.
 Hence  $L, L-E_3, E_7\in H_2(X)$ project to elements  in the quotient $H_1(\p V)$  of 
 orders $2,3,4$  respectively. Thus the map
 $H_2(X)\to H_1(\p V)$ is surjective, so that  $H_1(X\less V) = 0$.
 
 Now consider the commutative diagram induced by $\Phi_J$:
  $$
 \begin{array}{ccccccccc}
 0&\to& H_2(X\less V) \oplus H_2(V)& \to &H_2(X)&\to &H_1(\p V)&\to &0\\
& &\downarrow& &\downarrow&&\downarrow\\
 0&\to& H_2(Z\less \Nn) \oplus H_2(\Nn) &\to& H_2(Z)&\to&
  H_1(\p \Nn)&\to &0\end{array}
  $$
Since $\Phi_J$ is a homeomorphism $X\less V\to Z\less\Nn$ and $H_i(\Nn)  =0$ for $i>0$, we see that $H_2(Z)$ is generated by elements 
$(\Phi_J)_*E_3, (\Phi_J)_*E_7$, while the image of $H_2(Z\less \Nn)$  is generated by 
$6(\Phi_J)_*E_3$ (from  $3L-E_{123}$) and $4(\Phi_J)_*E_7$ (from $E_{4567}$).
 
 Since $X\less V\cong Z\less \Nn$, we know from above that 
  $H_1(Z\less\Nn)=0$.   A similar Mayer-Vietoris sequence argument shows that  $H_1(Z) = 0$.  Hence $H^2(Z\less\Nn)$ and $H^2(Z)$ are both free abelian groups and the map between them is dual to
the inclusion $H_2(Z\less\Nn)\to H_2(Z)$.  This completes the proof when $\ell=4$.

When $\ell=5$ one just needs to add a further blow up to the chain  $E_4,\dots, E_7$.  Thus the generators of  $H_2(X\less V)$ are $3L-E_{123}$
and $E_{4\dots 8}$.  The rest of the argument is essentially  the same.\end{proof}

\begin{cor}\labell{cor:H2} Suppose that $M_\ell$ exists and denote $Y: = H^{-1}(0)$ considered as the boundary of $H^{-1}([-6,0])$.
Then there is a unique class $e_Z(Y)\in H^2(Z;\Q)$ that restricts to the Euler class $e(Y)$ of the locally trivial $S^1$-bundle $Y|_{Z\less\bp}\to Z\less \bp$.  Moreover  if we identify $Z$ with $\C P^2_{1,2,3}\#
\ov{\C P}\,\!^2_{1,1,\ell}$ as in Lemma~\ref{le:T} then
 \begin{equation}\labell{eq:eY}
 [e_Z(Y)]|_{\C P^1_{2,3}} = -{\ts \frac 16}\,\tau_{2,3}, \quad 
 [e_Z(Y)]|_{{\C P}\,\!^1_{1,\ell}} = -{\ts \frac{1}{\ell}}\,\tau_{1,\ell}.
 \end{equation}
 \end{cor}
 \begin{proof}  The first statement is an immediate consequence of
  Lemmas \ref{le:T} and  \ref{le:H2}.  
  The second  also uses  the fact that 
 $e_Z(Y) = -\frac d{d\ka}[\tau_{\ka}]$ by
  Godinho's generalization of 
the Duistermaat--Heckmann formula.  Note that $e_Z(Y)$
  does restrict to an integral class on $Z\less{\bp}$ because the image of $H_2(Z\less\bp)$ in $H_2(Z)$ is generated by $6E_3, 4E_7$. 
 \end{proof}
 
\subsection{The symplectic topology  of the 
reduced spaces: preliminaries.}\labell{ss:Z}

By Corollary ~\ref{cor:H2}, the reduced space $(Z,\om_\ka)$ for $\ka\in 
(-\ell,\ell)$ is an orbifold blow up.  It can be constructed as a toric manifold  whenever 
$$
\ell+\ka < 3+\ka/2, \quad\mbox{ or equivalently }  -\ell < \ka <2(3-\ell)
$$
 since then the triangle 
$(1+\ka/\ell)T_{1,\ell}$ is a subset of $(1+\ka/6)T_{2,3}$. 
  
The following lemma is proved in \cite[Prop~1.6]{M}. (The argument is explained below.)

\begin{lemma}\labell{le:Z} 
 For all integers $\ell\in [2,6]$
and all $\ka \in (-\ell,\ell)$  there are symplectic orbifolds  $(Z,\om_\ka)$ 
satisfying the conditions in Lemma~\ref{le:T}.
\end{lemma}

\begin{rmk}\labell{rmk:ZT}\rm  In fact, if all we are interested in is 
existence then we do not need this result from
\cite{M}  because $[\om_0]$ is the anticanonical class $-K$ on $Z$. 
 Hence, provided that we give $Z$ a sufficiently generic complex structure $J_Z$,
 we can take $\om_0$ to be the K\"ahler form induced from
 projective space by the embedding given by sections of
 a suitable multiple of the anticanonical class;
 and
 then define $\om_\ka$ for $-\ell<\ka<0$ by decreasing the size of the exceptional divisor ${\C P}^1_{1,4}$, or, equivalently, by decreasing the size of the ellipsoid $\la
E(1,4)$ that is embedded in  $E(2,3)$.  This constructs  
$(Z,\om_\ka)$ for $-\ell<\ka\le0$.  The result for $\ka>0$ follows by symmetry.  More precisely, we will see in the proof of Proposition~\ref{prop:ell} given in \S\ref{s:3} below
that the diffeomorphism $\psi$ of Proposition~\ref{prop:psi}
covers a diffeomorphism $\psi_Z$ of $Z$ such that $\psi_Z^*([\om_{\ka}]) = -[\om_{-\ka}]$.

But notice that we do need $J_Z$ to be \lq\lq generic". 
In particular we cannot use the toric structure $J_T$ because this is not NEF;
for example when $\ell=4$ the edge in Figure \ref{fig:6} with conormal $(0,-1)$ pulls back to a line in $X_7$ 
in class $L-E_{4567}$
and $K\cdot(L-E_{4567})=1$. Suitable complex 
structures are constructed in 
Lemmas~\ref{le:J7} and \ref{le:J8}.\end{rmk}

The above remarks, together with Proposition~\ref{prop:Zuniq} below,
 are all that is needed  to construct 
$M_\ell$ as a symplectic, or indeed as a complex, manifold.  However,
to establish the uniqueness results, we need to know much more about the reduced spaces $Z$ than simply the existence of suitable symplectic forms.
We now adapt a definition from Gonzalez \cite{Gonz}.
The word \lq\lq rigid" is used here by analogy with the complex case.  
Symplectic forms can always be deformed, but in the rigid case these deformations have very little consequence, and the symplectic structure is essentially  unique.

\begin{defn}\labell{def:rig}  A symplectic orbifold $(Z,\om)$ 
 is called {\bf rigid} if the following conditions hold:
\SSS

\NI {\rm (a)} ({\bf Uniqueness.}) Any two cohomologous symplectic forms on $Z$ are diffeomorphic;\SSS

\NI {\rm (b)} ({\bf Deformation implies isotopy.})
 Every path $\om_t, t\in [0,1],$ of symplectic forms on $Z$ with  $[\om_1]=[\om_0]$ can be homotoped through
families of symplectic forms with the fixed endpoints $\om_0$ and $\om_1$ to an isotopy, i.e. a path $\om_t'$ such that $[\om_t']$ is constant;
\SSS

\NI {\rm (c)} ({\bf Connectness.}) 
For all symplectic forms $\om'$ on $Z$ the group
  $\Symp_H(Z,\om')$ of symplectomorphisms that act trivially on integral  homology is connected.
\end{defn}

Note that condition (c)  implies that the diffeomorphism in (a) is determined uniquely up to symplectic isotopy by its action on $H_*(Z;\Z)$.

\begin{rmk}\rm  To put our results on $Z$ in perspective, we observe that
the papers \cite{LM,Mdef} show that $(X_k, \om)$ satisfies the first two of these conditions; it satisfies (c)  when $k\le 3$ (cf. Lalonde--Pinsonnault \cite{LP} for the case $k\le 2$  and Pinsonnault \cite{Pin2} for $k=3$), the case $k=4$ is open,  but when $k\ge 5$  Seidel showed in \cite{Sei} that there are $\om'$ on $X_k$ for which (c) does not hold, by constructing symplectomorphisms that twist one Lagrangian sphere around another.   But none of these  Dehn twists can be constructed so as to descend to $Z$.
Hence the nonrigidity of these $X_k$  does not contradict 
the rigidity of $Z$.
 \end{rmk}
 
 \begin{lemma} The weighted projective space $\C P^2_{1,2,3}$  is
symplectically rigid.
\end{lemma}
\begin{proof}  Condition (b) is obviously satisfied  since
we can make arbitrary changes in the cohomology class by rescaling. The other conditions can be proved by adapting the arguments given below for $Z$.  Further details are left to the reader.
 \end{proof}
 
 The  proof of the following theorem 
 takes up the rest of this section.
 
 \begin{thm}\labell{thm:rigid}  The orbifold $Z = \C P^2_{1,2,3}\#
\ov{\C P}\,\!^2_{1,1,\ell}$ is symplectically rigid.
\end{thm}

We prepare for the proof by collecting together some useful technical results.  

\begin{figure}[htbp] 
   \centering
   \includegraphics[width=3in]{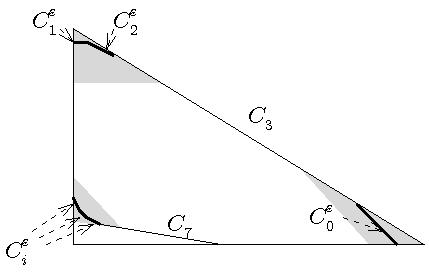} 
   \caption{The toric model for the curves $C_i$,
   where $0\le i<7, i\ne 3,$  in $\Cc$ when $\ell=4$, where
   we denote the edge representing $(C_i,\tau_\eps)$ 
    by $C_i^\eps$.  The moment image in $\De$ of the neighborhood $\Nn$ is shaded as in Figure \ref{fig:2}; 
   $\tau_\eps$ is the corresponding form on the small neighborhood $\Vv$ of $\Cc$.}
   \label{fig:6}
\end{figure}

The problem with the resolution $\Phi_J: (X_k,J)\to Z$
from the symplectic point of view  is that $\Phi_J$ is not a symplectomorphism; in particular, the pull back $\Phi_J^*(\om)$ of any  symplectic form on $Z$ is degenerate along $\Cc$.  However, we can deal with this as in \cite{M}, replacing $\Phi_J^*(\om)$ by a symplectic approximation as follows.  

By the local Darboux theorem explained 
 before Lemma \ref{le:phiZ}  every symplectic form on $Z$ 
 can be isotoped to be standard  in the neighborhoods 
$\Nn_m$ of the singular points.  Therefore, we shall 
only consider symplectic forms on $Z$ that are  standard  in 
$\Nn: = \cup\Nn_m$. 
Since the standard form is toric, we may identify the neighborhood of 
each singular point with a neighborhood of the appropriate vertex in the toric model $\De$ illustrated in Figure \ref{fig:2} for the case $\ell=4$.
The pullback of this standard form by $\Phi_J$ is 
degenerate along $\Cc$ but  is toric elsewhere in $\Vv$,
and clearly may be modified inside $\Vv$ to a form 
$$
\tau_\eps
$$
 that is toric and 
nondegenerate, and so that its (local) moment polytope is a neighborhood of the short edges $C_i^\eps$ in the approximation $\De_\eps$.   
Figure \ref{fig:6}  illustrates the case $\ell=4$.\footnote
{
In this figure we are illustrating a case in which $(Z,\om)$ has a global toric structure.  This does not hold for all $[\om_\ka]$;
but all that concerns us here is the local toric structure near the singular points, which always exists.}
  There is an analogous picture for $\ell=5$ with one extra short edge with conormal $(-1,-5)$.  The paper \cite{M} describes to how to construct such 
a symplectic approximation for any orbifold blow up of $\C P^2_{1,a,b}$.   In the language of that paper we are replacing the curves in $\Cc$ by the relevant parts of the inner and outer approximations to the ellipsoids $\la_1E(2,3)$ and  $\la_2E(1,\ell)$.  

Here are some useful properties of $\tau_\eps$:
\begin{itemize}
\item
As $\eps\to 0$, $\tau_\eps$ converges to $\Phi_J^*(\om)$ in $\Vv$.  \item The curves $C_i^\eps, C_{i+1}^\eps$ intersect $\tau_\eps$-orthogonally. (This is true for 
 the spheres corresponding to any pair of intersecting edges  of the moment polytope. For this is obviously true if the two edges 
lie along the coordinate axes through the origin. But, by the Delzant (smoothness) condition,
all vertices are affine equivalent to this one,
and affine transformations of the moment polytope lift to symplectomorphisms of the toric manifold.)
\item  We may recover $\om$ near $\bp$  from $\tau_\eps$ 
(and hence 
$\Phi_J^*(\om)$ near $\Cc$) by 
the blowing down process  described in Symington \cite{Sy}; see also \cite[Lemma~2.3]{M}.  This is a generalized symplectic summing process that first removes the curves $\Cc$ and then adds a suitable standard contractible open set.
\end{itemize}

In the following discussion we shall allow ourselves to decrease $\eps$ and shrink the sets $\Nn,\Vv$ as necessary.
Note also that {\it  this local toric model may be extended to include 
the $-1$ curves $C_3$ and $C_k$.}  
 Thus we shall assume that $\tau_\eps$ is nondegenerate on $C_3$ and $C_k$ and that these intersect $\Cc$ 
orthogonally with respect to the symplectic form. The symplectic neighborhod theorem then implies  that
\begin{itemize}
\item  $\tau_\eps$ is uniquely determined near $\Cc\cup(C_3\cup C_k)$ by its cohomology class.
\end{itemize}
Given $\eps>0$ we shall further assume that $\int_{C_3}\tau_\eps = 1-3\eps$, so that   $\tau_\eps$ integrates to $1$ over 
$\Phi^{-1}(D_1) = \cup_{i=0}^3C_i$.  However there is a choice for the size of $C_k$.  If it is necessary to emphasize this, we shall denote the form that integrates to $\la - (\ell-1)\eps$ over $C_k$ by 
$\tau_{\la,\eps}$. Then 
$$
\int_{\Phi^{-1}(D_2)} \tau_{\la,\eps}= \sum_{i=4}^k\int_{C_i}
\tau_{\la,\eps} = \la.
$$
Thus, in the notation of equation (\ref{eq:tau}),
$$
[\tau_{\la,\eps}] = 3a - (e_1+e_2+e_3) - \la(\sum_{i=4}^k e_k) + O(\eps).
$$
We shall also suppose that $0< \eps \ll \la\le 1$.
More precisely, we choose $\eps> 0$ so small that
\begin{equation}\labell{eq:min}
\mbox{{
the minimum of }} \int_E\tau_{\la,\eps}\;\;\mbox {{ for }} E\in \Ee(X_k)\;
\mbox {{ is assumed on the class }}
E_k.
\end{equation}
 (That this is possible
can be directly checked using the description of $\Ee(X_k)$ given in
equation (\ref{eq:Ek}).)  
 
\begin{defn} We 
denote by $\tau_{\eps,\la}$ (simplified to $\tau_\eps$) the toric symplectic form on $\Vv \cup {\rm nbhd}\,(C_3\cup C_k)$ described above.
Fix a $\tau_\eps$ compatible complex structure $J_\Vv$ on $\Vv$ for which the curves in $\Cc$ are holomorphic, and  let $\Om$ be any symplectic form on $X_k$ that equals $\tau_\eps$ in $\Vv$. 
Then
we define
\begin{itemize}
\item $\Jj_\Vv(\Om)$ to be the set of  $\Om$-tame almost complex structures that equal $J_\Vv$ near $\Cc$.

\item $\Jj_{\Vv, reg} (\Om)$ to be the subset of $J\in \Jj_\Vv(\Om)$
for which every $J$-holomorphic curve $u: S^2\to X_k$ whose image 
intersects $X_k\less \Vv$ has $c_1(u_*[S^2]) > 0$.
\end{itemize}

\NI Further, we say that a class $A\in H_2(X_k)$ is 
{\bf smoothly $J$-representable}
if it has a 
$J$-holomorphic and smoothly embedded representative.
\end{defn}

Here is the key technical lemma.

\begin{lemma}\labell{le:tech} 
{\rm (i)} The subset $\Jj_{\Vv, reg} (\Om)$ has second category in  $\Jj_\Vv(\Om)$ and  is path connected.  For every 
$J\in \Jj_{\Vv,reg}(\Om)$ the class $E_3$ 
is smoothly $J$-representable.
\SSS

\NI {\rm (ii)} {\rm [Pinsonnault]} Suppose that  $\Om|_{\Cc\cup C_3\cup C_k} = \tau_{\la,\eps}$ for some $\la \le 1$.
Then, for every  $J\in \Jj_\Vv(\Om)$ the class $E_k$ 
is smoothly $J$-representable.
\end{lemma}
\begin{proof}  To prove (i), 
recall that the moduli space $\Mm(A,J)$ of $J$-holomorphic 
maps $u:S^2\to X_k$ 
in class $A$ has expected (real) dimension $4 + 2c_1(A)$ and that, if 
$\Mm(A,J)$ is nonempty and consists of regular curves,
 this  
must be $\ge 6$, the dimension of the reparametrization group. 
But, standard results  (cf. \cite[Ch~3]{MS}) imply that
for each $A$ there is a subset of $\Jj_\Vv(\Om)$ of second category consisting of $J$ for which every $A$-curve that intersects $X\less\Vv$ is regular.  Since the set of classes $A$ is  countable, such $J$ lie in the set  
we have called $\Jj_{\Vv, reg} (\Om)$.  
 This proves the first statement in (i).  To prove that $\Jj_{\Vv, reg} (\Om)$ is path connected, recall that  any two elements in
$\Jj_{\Vv, reg} (\Om)$ can be joined by a generic path  
consisting of elements $J_t$ for which the cokernel of the linearized Cauchy--Riemann operator $D_u$ has dimension at most $1$.  Therefore, 
if $\Mm(A,J_t)\ne \emptyset$, we have $4 + 2c_1(A)\ge 5$, which implies that $c_1(A)>0$.  Thus $J_t\in \Jj_{\Vv, reg} (\Om)$ for all $t$.  

Finally, since $E_3$ has nonzero Gromov--Witten invariant, it is  represented
by some $J$-holomorphic stable map for all $J\in \Jj_{\Vv} (\Om)$.  Let $B_1,\dots, B_m$ be the classes of its components.  Since $c_1(E_3) = 1$, if this stable map is not smooth at least one of these components must have $c_1(B_i)\le 0$. But this is impossible when $J\in \Jj_{\Vv, reg} (\Om)$.  This proves (i).

We prove 
statement (ii)  by using Lemma 1.2 in Pinsonnault \cite{Pin}, which states that 
a class $E\in \Ee_k$  whose symplectic area is minimal among all the classes in $\Ee_k$  has an smooth $J$-representative for {\it all} tame $J$.
This applies here since $E_k$ is such a minimal class by construction.
 \end{proof}

\subsection{The rigidity of $Z$.} 

\begin{lemma} \labell{le:rig} $Z$ satisfies condition (b) in the definition of rigidity.
\end{lemma}
\begin{proof}
We shall prove this for the  reduced spaces $(Z,\om_\ka)$ 
where $-\ell< \ka \le 0$.  The case $0< \ka <\ell$ follows by symmetry; cf. Remark~\ref{rmk:ZT}.   

Suppose we are given two cohomologous symplectic forms $\om_0, \om_1$  on $Z$ that are connected by a deformation $\om_t$. First multiply each $\om_t$ by a suitable constant so its integral over the divisor 
$D_1=\C P^1_{2,3}$ is constant and equal to $1$.  Next, use
 a parametrized version of the local Darboux theorem of \S2 to
ensure  that
each of these forms is standard in some neighborhood of the singular set $\bp$.  Then the pullback family $\Phi_J^*(\om_t)$ on $X_k$ is 
constant on the neighborhood $\Vv$.
We claim that by a relative version of the \lq\lq deformation implies isotopy" result from \cite{Mdef}
we can
 homotop the deformation
 $\Phi_J^*(\om_t)$  in $X_k$ to an isotopy, keeping the endpoints fixed and also not changing the forms near $\Cc$.   Once this is done, we can push forward the resulting isotopy by $\Phi_J$ to an isotopy in $Z$.
 
 To establish the claim, several remarks are in order.  
 \SSS

 \NI$\bullet$  Since we are keeping the forms
 fixed in $\Vv$ it does not matter that they are not symplectic 
 along $\Cc$. (Alternatively, we can can change them in $\Vv$ to equal $\tau_\eps$.)\SSS
  
 \NI$\bullet$  One changes a deformation to an isotopy by inflating along certain 
 symplectically embedded curves $S$.  
 If these curves do not intersect $\Cc$ then this inflation process will not change the forms near $\Cc$. 
 The general inflation process is described in \cite{Mdef}; some special cases are described in \cite{M}.
 \SSS
 
 \NI$\bullet$   We can insure that the
 curves $S$ do not intersect $\Cc$ by choosing them to lie in classes
in $H_2(X_k\less \Cc)$ and also to be $J$-holomorphic for some $J$ 
for which the curves in $\Cc$ are holomorphic.
\SSS

The above remarks apply to the general relative 
inflation process. 
In fact in our case
 $H_2(X_k,\Cc)$ is generated by the classes $E_3$ and $E_k$ that project to the divisors $D_1,D_2$, and we have already arranged that the forms 
 $\Phi_J^*(\om_t)$ have the same integral over $E_3$.  Hence we only need to worry about the size of $E_k$.  If this is too big, it is easy to decrease it, essentially by inflating along the representative $C_k$ of $E_k$ itself.
 (As pointed out by Li--Usher \cite{LU}, one can also interpret inflation as a symplectic connect sum, and hence can inflate along
curves of negative self intersection.)
   If it is too small and $\ka\le0$ we can increase it to 
   $\ka$  by inflating along a curve in the class 
    $A_k$ where
$$
A_7: = 5(L-E_{123}) -6E_{4\dots k},\quad 
A_8: = 11(L-E_{123}) -12 E_{4\dots k}.
$$
Then 
$$
A_k\cdot A_k > 0, \quad d(A_7) = {\ts \frac 12}(A\cdot A + c_1(A)) = 6,
\quad d(A_8) = 6.
$$
These inequalities imply that the Gromov--Witten invariant $Gr(A_k)$ that counts embedded holomorphic curves through $d(A_k)$ points is nonzero,
so that these classes  have smooth  $J$-representatives for generic $J$.   Moreover,   
because $PD(A_k)$ is a multiple of $[\om_\mu]$ for some $\mu>0$, $A_k$ 
 can be used to change the cohomology class of $\om_\ka$ by increasing $\ka$ to any number $< \mu$.  In particular, we can increase $\ka$ to $0$.
 \end{proof}

\begin{prop}\labell{prop:Zuniq} Any two cohomologous symplectic forms on $Z$ are diffeomorphic.
\end{prop}
\begin{proof}   The  argument below basically shows that any symplectic form $\om$
on $Z$ is the blow up of a form $\rho$ on $\C P^2$, so that the result  follows from the uniqueness of 
symplectic structures on $\C P^2$. However, it is easiest to explain
the details if we start with two symplectic forms $\om', \om''$ on $Z$.
We assume as we may that these agree on the neighborhood $\Nn$ of the singular points $\bp$.

We shall work on the resolution $X_k$.
Denote by $\Om_\eps'$ and $\Om_\eps''$ 
the symplectic forms on $X_k$ obtained from $\Phi_J^*(\om')$
and $\Phi_J^*(\om'')$  by changing them in $\Vv$ to equal $\tau_\eps$.
We will show that $\Om_\eps'$ is diffeomorphic to $\Om_\eps''$, by a diffeomorphism $\phi$ that equals the identity in a neighborhood $\Vv_1\subset \Vv$ of $\Cc$.  
Since  we may recover $\om', \om''$ from $\Om_\eps', \Om_\eps''$
by the same symplectic blow down process near  $\Cc$, we may extend this diffeomorphism by the identity to get the desired diffeomorphism of $Z$.

To construct this diffeomorphism of $X_k$, choose $J'\in 
\Jj_{\Vv, reg} (\Om_\eps')$, and consider the
 corresponding $J'$-holomorphic spheres
$C_3'$ and $C_k'$ in classes $E_3$ and $E_k$ (which exist by 
Lemma~\ref{le:tech}.)  Our first aim is to extend the local toric model to include these curves $C_3', C_k'$.
 By positivity of intersections, these must  each intersect $\Cc$ transversally. 
In fact, because $C_3'\cdot C_0=1$, 
$C_3'$ meets $C_0$  transversally at a single point $q_0'$.  Similarly,
$C_3'$ meets $C_2$   at   $q_2'$, $C_k'$ meets $C_{k-1}$ at $q_k'$ 
and there are no other intersections.   Let $q_2,q_0,q_k$ denote the corresponding points of intersection of $C_3\cup C_k$ with $\Cc$.
We now claim that 
there is an $\Om'_\eps$-symplectic isotopy $g_t, t\in [0,1],$
supported near $\Cc$  such that
$
g_0=id,$ $g_t(\Cc) = \Cc$ for all $t$,  and so that
$g_1(C_3')$ and $g_1(C_k')$  coincide with $C_3$ and $C_k$ near $\Cc$. 
 To achieve this, we first isotop $C_3'$ and $C_k'$ so that they meet
 $\Cc$ at the correct points, then straighten them out so that the intersection is orthogonal,\footnote
{
This straightening technique is well known: see for example
 the proof of Theorem 9.4.7 (ii) in \cite{MS}.} and finally isotop them to coincide with $C_3, C_k$ near the points $q_i$.

Therefore we may assume that the $J'$-holomorphic spheres
$C_3'$ and $C_k'$  are such that
 $C_3'=C_3$ and $C_k'=C_k$ near their intersections with $\Cc$. We denote by $\Ss'$ the set of symplectic forms on $X_k$ that equal some toric form near $\Cc\cup C_3'\cup C_k'$.
 Similarly,
 we may choose $J''\in \Jj_{\Vv, reg} (\Om_\eps'')$ such that the corresponding spheres $C_3'', C_k''$ equal $C_3, C_k$ near $\Cc$.
 For the other $i$ we define $C_i': = C_i = :C_i''$.
 
 Now consider the set  of curves $\cup_{i=1}^kC_i'$, i.e. all the curves except for $C_0$ in class $L-E_{123}$.  This set has two 
 connected components $\Cc_1': = \cup_{i=1}^3C_i'$ and $\Cc_2': = \cup_{i=4}^kC_i'$ both with local toric models. 
 The symplectic blow down process as described in Symington \cite{Sy}
 removes this set of curves, inserting their stead two closed regions
 $\Rr_1, \Rr_2$ whose boundaries collapse to  $\Cc_1',\Cc_2'$ under the characteristic flow; see Figure \ref{fig:new}.   Denote this symplectic blow down by 
 $(Y', \om_Y')$.   It contains (embedded copies of) the regions $\Rr_j, j=1,2,$ and one can get back to $X_k$ by cutting  out the interior of these regions and collapsing their boundaries.
 
  \begin{figure}[htbp] 
    \centering
    \includegraphics[width=3in]{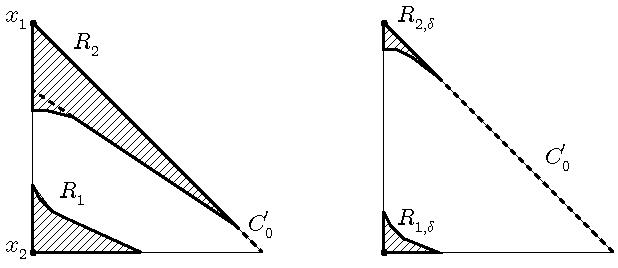} 
    \caption{The regions $\Rr_j$ that we blow down to get $Y$, in their original form on the left and made small on the right.}
    \label{fig:new}
 \end{figure}

 Let  $x_1',x_2'$  be the points in $Y'$ corresponding to 
 the vertices $x_1,x_2$ of the simplex in Figure \ref{fig:new}.  
 Notice that by varying the size of the curves $C_i', i\ge 1,$ (i.e.
 by moving $\Om'_\eps$ along a path in $\Ss'$)  we can make the regions $\Rr_j$ arbitrarily small.  (This is not a question of making $\eps$ smaller, since that  decreases $C_0$, but rather 
 of decreasing the size of all curves in $\Cc_1'\cup \Cc_2'$, increasing $C_0$ correspondingly.)  Thus given any neighborhoods of $x_1',x_2'$ in $Y'$ we can 
 construct a symplectic form $\Om_\de'$ on $X_k$ that lies in $\Ss'$ by removing suitably small copies $\Rr_{j,\de}$ of $\Rr_j, j=1,2$ in these neighborhoods.
 Clearly, such a form $\Om_\de'$ can be deformed to the original form $\Om_\eps'$ within $\Ss'$.

Next observe that $H_2(Y') = \Z$, with generator represented by a symplectically embedded $2$-sphere $S_0'$ through $q'$,
 the image of $C_0'$.
 Therefore by Gromov's well known theorem (cf. \cite[Ch.~ 9.4]{MS}), $(Y', S',\om_Y')$ is symplectomorphic to $(\C P^2, \C P^1)$ with its standard form.
 Applying a similar argument to the curves $\cup_{i=1}^kC_i''$, we get another copy $(Y'', \om_Y'')$ of $\C P^2$, that also contains embedded copies of $\Rr_j, j=1,2.$  It 
 follows that there is a symplectomorphism  
 $$
\psi:  (Y', S_0', x_1', x_2',\om_Y')\;\;\to\;\; (Y'', S_0'',x_1'',x_2'',\om_Y'').
 $$
 We can isotop $\psi$ so that, for sufficiently small $\de$ and for $j=1,2$, it takes the copy of $\Rr_{j,\de}$ in $Y'$ to that in $Y''$. Then $\psi$ lifts to a symplectomorphism $(X_k,\Om_\de')\to (X_k,\Om_\de'')$ on the blow up.  Moreover, 
it can be chosen to be the identity near $\Cc$ since the curves $C_3',C_3''$ and $C_k', C_k''$ coincide near $\Cc$.   Hence $\Om_\eps'$ is deformation equivalent to $\psi^*(\Om_\eps'')$ by a deformation in $\Ss'$.  One now changes this deformation to an isotopy 
in $\Ss'$ by the inflation procedure as described in 
Lemma \ref{le:rig}.
\end{proof}

\begin{rmk}\rm  If in the above proof we only had to deal with the singularity at $p_\ell\in Z$ (which is resolved by the curves in $\Cc_1'$) then we
could perform an orbifold blow down directly from $Z$, with no need to pass to $X_k$.  However, the other two singularities at $p_2,p_3$ do not have such a direct blow down, and we must first blow up to $X_k$ before passing to the blow down.
\end{rmk}

To prove that $Z$ satisfies condition (c), we first consider the case of the reduced space $Z$ at level $\ka \in (-\ell,2(3-\ell)).$
We saw at the beginning of \S\ref{s:3} that in this case
 $(Z,\om_\ka)$ has a toric structure with moment polytope as pictured in Fig \ref{fig:2}.    
\MS

\begin{lemma}\labell{le:toric} If  $(Z,\om_\ka)$ is toric then
the group $\Symp_H(Z,\om_\ka)$ is connected. 
\end{lemma}  
\begin{proof} Choose $\la = 6(\ell+\ka)/\ell(6+\ka)$, the ratio of the integral of $\om_\ka$ over $D_2$ to its integral over $D_1$; cf. equation (\ref{eq:omZ}).
Since $(Z,\om_\ka)$ is toric, the local toric form $\tau_{\la,\eps}$
extends to a global toric form $\Om$ on $X_k$ which we may assume to equal
a multiple of $\om_\ka$ outside $\Vv$.

Suppose that $\phi\in \Symp(Z,\om_\ka)$. Since $\phi$ 
acts trivially on homology, parts (ii) and (iii) of
   Lemma \ref{le:phiZ} imply that
$\phi$ is symplectically isotopic to a symplectomorphism $\phi_1$ that is the identity near the singular points $\bp$. \MS

\NI {\bf Step 1:}  {\it   $\phi_1$ is symplectically isotopic to a symplectomorphism $\phi_2$ that is the identity near the divisors $D_1$ and $D_2$.}

Note that any symplectomorphism $\phi$ of  $Z$ that is the identity near the 
singular points $\bp$
lifts to a symplectomorphism $\Tilde\phi$ of $(X_k,\tau_\eps)$ that is the identity on some neighborhood $\Vv$ of  $\Cc$; cf. Figure \ref{fig:4}.

By Lemma~\ref{le:tech}, there is a path $J_t\in \Jj_{\Vv,reg}(\Om)$ from $J_\Vv$ to $J_1: = (\Tilde\phi_1)_*(J_\Vv)$.  Let $C_{3,t}, C_{k,t}$ be the  $J_t$-holomorphic  
representatives of the classes $E_3, E_k$. 
They are symplectically embedded by construction, and  as in the previous proof we may alter them by a symplectic isotopy supported near $\Cc$ to families of curves that coincide with $C_3=C_{3,0}$ and 
$C_k=C_{k,0}$ near $\Cc$ for all $t$.  Then,  the symplectic isotopy extension theorem implies that there is a family of symplectomorphisms $g_t$ with $g_0=id$, and such that, for all $t$, 
$$
g_t=id \mbox {  near }\Cc,\;\; g_t(C_3) = C_{3,t},\;\; g_t(C_k) = C_{k,t}.
$$
Then $ \Tilde\phi_2: = g_1^{-1}\circ \Tilde\phi_1$  takes $C_3$ to itself and $C_k$ to itself, and is the identity near the points 
of intersection with $\Cc$.

Since $C_k$ intersects $\Cc$ in a single point,  it is easy to adjust the isotopy $g_t$ so that $\Tilde\phi_2 = id$ on $C_k$.  However,
$C_3$ meets $\Cc$ in two points and so the induced map on 
$C_3\less \Cc$ may not be isotopic to the identity by an isotopy of compact support, although there is an isotopy to the identity 
 that fixes $C_3$ near the point 
$C_3\cap C_2$ and rotates $C_3$ near $C_3\cap C_0$.  On the other hand, it is not essential to consider only those symplectomorphisms that are the identity on $\Cc$ since all we need is that
the symplectomorphisms on $X_k$ descend to symplectomorphisms on $Z$.
Therefore we just need to check that there is an $S^1$ action near the singular point $p_2$ in $Z$  that lifts to an action near $C_0$ that fixes the point $C_0\cap C_3$ but rotates both $C_0$ and a neighborhood of
the point $C_0\cap C_3$ in
$C_3$.    But this is clear because  our local models are toric and there is a suitable $S^1$ subgroup of the torus $T^2$. 

Hence we may assume that $\Tilde\phi_2$ is the identity on 
 $C_3\cup C_k$ as well as near $\Cc$ and then make a final 
 isotopy in the directions normal to $C_3\cup C_k$ to make it the identity on a neighborhood.   This gives the desired isotopy in $X_k$.
Since all the symplectomorphisms  considered  are either equal to the identity near $\Cc$ or equal to a rotation that is contained in the local torus actions,  they push forward to $Z$, yielding the desired isotopy of $\phi_1$ to $\phi_2$.  
\MS
  
  \NI
  {\bf Step 2:}  {\it We may isotop $\phi_2$ to a
  symplectomorphism  $\phi_3$ that is also the identity near 
  the divisor $D_3$ represented by the edge $v_1v_2$ in the toric model
  of Figure \ref{fig:5}.}
  
  This edge pulls back to a curve $C'$ in $X_k$ in the class $L-E_{14}$.
  This class is again in $\Ee(X_k)$ and so has a smooth representative for $J\in \Jj_{\Vv,reg}(\Om)$.    Therefore this step may be accomplished by arguing as in Step 1.   
  Note that again $C'$ intersects $\Cc$ in two points.  Therefore to make $\phi_3=id $ on $D_3$ we may 
  need to rotate $Z$ near  its singular point of order $k$.
  This is possible as before.\MS 
  
     \begin{figure}[htbp] 
   \centering
   \includegraphics[width=4in]{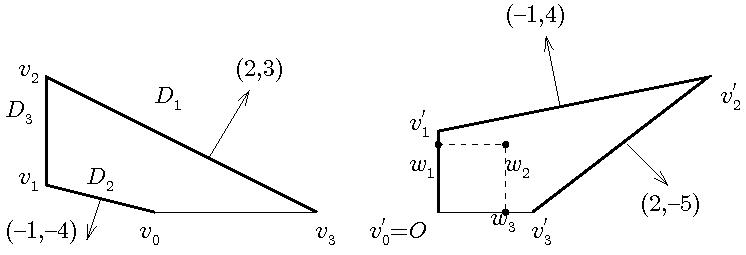} 
   \caption{The original moment polytope is on the left and its affine
  image is on the right.  The three divisors $D_1,D_2,D_3$ are represented by the thicker lines.}
   \label{fig:5}
\end{figure}

  \NI {\bf Step 3:} {\it Shrinking the support of $\phi_3$.}
 
  We have now isotoped $\phi$ to a symplectomorphism $\phi_3$ that  is the identity near the divisors
  $D_i, i=1,2,3,$  represented by the three edges $v_2v_3, v_0v_1$ and $v_1v_2$ of the moment polytope $\De$. Because $v_0$ is a smooth point of the polytope, 
  we may by an affine change of coordinates identify $\De$ with the polytope
  $O=v_0',v_1',v_2',v_3'$ where $O$, the image of the point $v_0$, is at the origin; cf. Figure \ref{fig:5}. 
 It follows that the open set $\Uu_Z= Z\less (D_3\cup D_1)$,  the inverse image under the moment map of $\De\less ( v_1v_2\cup v_2v_3 )$, has a natural
 Darboux chart whose image $\Uu_0$ is the  open convex subset of $(\C^2,\om_0)$
 defined by the two equations  
 $$
- |z_1|^2 + 4|z_2|^2 < c_1,\quad 2|z_1|^2 - 5|z_2|^2 < c_2,
 $$
 for some $c_i>0$.  (Note that the coefficients in these equations are given by the conormals to the edges $v_1'v_2', v_2'v_3'$.) Moreover, in these coordinates, the divisor $D_2$ corresponds to the 
 disc $z_1=0$.

 For $0<\la < 1$ let $m_\la$ be the image in  $\Uu_ Z$ of the rescaling map $\Uu_0\to \Uu_0$ given by multiplication by $\la$.
 Since $\phi_3$ has support in $\Uu_Z$ the symplectomorphism 
 $$
 \phi_{3+t}: = m_{1-t}\circ \phi_3\circ m_{1-t}^{-1}
 $$
  is well defined for all $t\in [0,1)$ and has support in $m_{1-t}(\Uu_Z)$.  In particular it is the identity on $D_2$ for all $t$.  Moreover, for $t$ sufficiently close to $1$ its support maps into a square of the form 
  $Ow_1w_2w_3$.   Thus its support is contained in the interior of a 
   subset of $(Z,\om_\ka)$ symplectomorphic to the product  $P: = \bigr(S^1\times [0,1]\bigl)\times D^2$ with a product symplectic form $\om_0$ that has the same integral over the two factors. 
  
  Now denote by $\Symp(P,\p P;\om_0)$ the 
group of symplectomorphisms of $(P,\om_0)$ that are the identity near the boundary.  Since the first three steps isotop $\phi$ to an element of 
$\Symp(P,\p P;\om_0)$,
the  following step completes the proof. 
 \MS
 
 \NI {\bf Step 4:}  {\it $\Symp_0(P,\p P;\om_0)$ is contractible.}\SSS
 
We may identify this group with the subgroup
$$
\Gg_0= \Bigl\{g\in \Symp(S^2\times S^2,\si\times\si)\,|\, g=id.\mbox{ near } S^2\times \{0\}\cup \{0,\infty\}\times S^2\Bigr\},
$$
of the group $\Gg$ of symplectomorphisms of $(S^2\times S^2, \si\times \si)$ that are the identity on $(\{0\}\times S^2)\cup (S^2\times \{0\})$. It follows from work of Gromov  that $\Gg$ is contractible; see the survey article \cite{LM1} or \cite[Ch~9.5]{MS}. Moreover,
there is a fibration sequence
$$
\Gg_0\to \Gg\stackrel{ev}\to {\mathcal Emb}
$$
where ${\mathcal Emb}$ is the space of symplectic embeddings
$\ov g:\{\infty\}\times S^2\to S^2\times S^2$ that 
extend to  elements of  $\Gg$.  (Notice that the fiber of $\ev$  consists in fact
of maps that are the identity {\it on} $\{\infty\}\times S^2$ and {\it near} 
the point $\{\infty\}\times \{0\}$, but  not {\it near} the whole of 
this sphere.   But a standard Moser argument shows that the space of such maps is homotopy
equivalent 
to $\Gg_0$.)
Because the two $2$-spheres in $S^2\times S^2$ have the same size, it follows as in \cite{LM1,MS} that
${\mathcal Emb}$  is homotopy equivalent to the contractible  space of $(\si\times \si)$-tame almost complex structures
on $S^2\times S^2$ that equal the product structure near 
$S^2\times \{0\}\cup \{0,\infty\}\times S^2$.   Therefore ${\mathcal Emb}$ and hence also $\Gg_0$ is contractible. \end{proof} 

\begin{lemma}\labell{le:Zconn}
 $\Symp_H(Z,\om_\ka)$ is connected for all $\ka \in(-\ell,0]$.
 \end{lemma}
 \begin{proof} 
We shall suppose that $\ka$ is too large for $(Z,\om_\ka)$ to be toric, since otherwise there is nothing to prove.
 Denote by $\Diff_0^c(Z\less\bp)$
 the identity component of the group of compactly supported diffeomorphisms 
 of $Z\less \bp$.  We give this (and all similar spaces) the usual direct limit topology so that the elements in any compact subset of $\Diff_0^c(Z\less\bp)$ all equal  the identity on some fixed neighborhood of $\bp$.
 Similarly, let  $\Symp^c(Z\less\bp,\om_\ka)$ be the subgroup of 
 compactly supported elements of $\Symp(Z\less\bp,\om_\ka)$.
 Elements of this group must fix the classes 
 $e(Y), [\om_\ka]$ by Lemma~\ref{le:phiZ} and so act trivially on homology.
 
  We will assume as we may that $\om_\ka$ is standard in some
   neighborhood of $\bp$, and denote  by
  $\Om$ the symplectic form on $X_k$ that equals $\tau_\eps$ on
  $\Vv: = \Phi_J^{-1}(\Nn)$ 
  and  equals 
 $ \Phi_J^*(\om_\ka)$ on $X_k\less \Vv$.
  \MS
  
  \NI {\bf Step 1.} \,{\it It suffices to show that  $\Symp^c(Z\less\bp,\om_\ka)$ is path connected.}\SSS
  
  As in the second paragraph of the proof of Lemma~\ref{le:toric}, every element in  $\Symp_H(Z,\om_\ka)$  is isotopic to 
  a symplectomorphism that is the identity near $\bp$.
  
  \MS

  \NI {\bf Step 2.}\, $\Symp^c(Z\less\bp,\om_\ka)\subset \Diff_0(Z,\bp)$.\SSS
  
 Every $\phi\in \Symp^c(Z\less\bp,\om_\ka)$ lifts to a symplectomorphism 
 $\Tilde\phi$ of 
 $(X_k,\Om)$ that is the identity in $\Vv$.  As in Step 1 of the proof of Lemma~\ref{le:toric},  there is a path  $\Tilde\phi_t\in \Symp(X_k,\Cc,\Om)$ starting at $\Tilde\phi$ and ending at an element $\Tilde\phi_1$ that is the identity in some  neighborhood $\Nn(C_k)$ of  $C_k$.  Then change the symplectic form $\Om$ in $\Nn(C_k)$, 
 decreasing the size of $C_k$, to a form $\Om'$ that lies in a class with a toric representative.    Then $\Tilde\phi_1$ preserves the form $\Om'$ and so is the lift of an element $\phi_1$ in    
 $\Symp(Z,\om_{\ka'})$, where $\om_{\ka'}$ is homologous to a toric form. Therefore $\om_{\ka'}$ is diffeomorphic to a toric form by Proposition~\ref{prop:Zuniq} and we can apply Lemma~\ref{le:toric} to conclude that $\phi_2$ is smoothly isotopic to the identity.  Since $\phi$ is smoothly isotopic to $\phi_1$ by construction, this proves Step 1.
 \MS
 
  \NI {\bf Step 3.} {\it $\Symp^c(Z\less\bp,\om_\ka)$ is path connected.}\SSS
   
 Denote by $\Ss\Ff(\ka)$  the 
 space of all symplectic forms that are isotopic to $\om_\ka$ through a family $\om_{t,\ka}, t\in [0,1],$ of symplectic forms that are all standard in some fixed neighborhood of $\bp$.
Since $\Diff_0^c(Z\less\bp)$ acts transitively on $\Ss\Ff(\ka)$, there is a 
 fibration sequence
 $$
 \Symp^c(Z\less\bp,\om_\ka)\;\to\; \Diff_0^c(Z\less\bp)\;\stackrel{\al_\ka}\to \;\Ss\Ff(\ka).
 $$
It suffices to show that the map $(\al_\ka)_*:\pi_1(\Diff_0(Z,\bp))\to \pi_1(\Ss\Ff(\ka))$ is surjective.  By Lemma~\ref{le:toric} there is $\ka'<\ka$ for which this holds.
Therefore it suffices to construct a map $r: \Ss\Ff(\ka)\to \Ss\Ff(\ka')$ such that
$\al_{\ka'} $ is weakly homotopic to $r\circ \al_\ka$.

 To this end, 
 we use some ideas from \cite{Mac}. (As explained at the end of the proof, this approach gives somewhat more than we need.) 
  Denote by $J_\Nn$  the image of $J_\Vv$ under the blow down map $X_k\to Z$.
 Consider the space $\Aa(\ka)$ of all almost complex structures $J$ on $Z$ 
 that are  tamed by some form in $\Ss\Ff(\ka)$ and are equal to $J_\Nn$ near $\bp$.
Further, define $\Xx(\ka)$ to be the space of all  pairs $(\om,J)\in \Ss\Ff(\ka)\times \Aa(\ka)$ such that $\om$ tames $J$.
Then the projection map $\Xx(\ka)\to \Ss\Ff(\ka)$ has contractible fibers and so is a homotopy equivalence.  A similar statement holds for the projection $\Xx(\ka)\to \Aa(\ka)$.  (Because everything is normalized near $\bp$, the singular points  cause no problem.)

We now claim that $\Aa(\ka) = \Aa(\ka')$ for all $\ka'<\ka$.
This holds by Lemma~\ref{le:tech} (ii).  For every $J\in \Aa(\ka)$  
there is a unique embedded $J$-holomorphic curve $C_J$ in class $E_k$.
If $J$ is tamed by $\om\in \Ss\Ff(\ka)$, $\om$ is nondegenerate on $C_J$ and therefore we can inflate $\om$ along $C_J$, constructing a family of forms $\om_\la, \ka'\le \la\le \ka,$ that 

\begin{itemize}
\item tame $J$,
\item equal $\om$ away from $C_k$  and 
\item
are such that
$\int_{C_k}\om_\la = \la$.  
\end{itemize}
For details, see \cite{Mac}.  
(The argument needed for this is a little more delicate than in the usual inflation procedure since 
the forms $\om_\la$ must tame $J$.)

This argument shows that the spaces $\Ss\Ff(\ka)$ and $\Ss\Ff(\ka')$ are homotopy equivalent.  Further, we can define a map
$r:\Ss\Ff(\ka)\to\Ss\Ff(\ka')$ that induces this equivalence and is
 unique up to homotopy, as follows: given a compact family $\Mm = \{\om_\mu\}$ of  elements of $\Ss\Ff(\ka)$ choose a corresponding family $J_\mu$ of $\om_\mu$ tame almost complex structures, and then alter the
$\om_\mu$ appropriately near the curves $C_{J_\mu}$ to a family $\om_{\mu,\la}$.   There are choices here, but they are equivalent up to homotopy.  

Note finally that if all we aim to do is construct this map $r$ we can use the less delicate version of inflation: there is no need to insist that the modified forms
$\om_{\mu,\la}$ are $J_\mu$ tame.  Also, if we are only interested in $\pi_1$ we can restrict to one dimensional families $\Mm$.

Since
$\al_{\ka'} $ is clearly weakly homotopic to $r\circ \al_\ka$, this completes the proof.  
\end{proof}
 
 \begin{rmk}\rm The argument in Step 3 above shows that
 the homotopy type of the group $\Symp_0^c(Z\less\bp,\om_\ka)$ is independent of $\ka\in (-\ell,0]$.    In contrast, the homotopy type of the groups 
 $\Symp(\C P^2\#\ov{\C P}\,\!^2, \om_\la)$ vary with the cohomology class of the form $\om_\la$.  However, the one point blow up of $\C P^2$ is the unique manifold for which Pinsonnault's result 
 quoted in Lemma~\ref{le:tech} (ii) fails to hold.
 \end{rmk}

\begin{rmk}\rm Although we have carried
 out the proof of Theorem~\ref{thm:rigid} for the orbifold $Z$, most of the arguments 
  apply much more widely.  For example, the uniqueness statement of Proposition~\ref{prop:Zuniq} easily extends to the case when  $Z$ is an orbifold blow up of the weighted projective space
 $\C P^2(1,a,b)$ at $k$ distinct points, provided that $a,b$ are 
relatively prime.   (If $a,b$ are not relatively prime then the singularities of
 $\C P^2(1,a,b)$ are no longer isolated points,  and one would have to use a different kind of resolution.)
 One might even be able to extend it  further 
 (for example to  blowups of any $\C P^2(a,b,c)$), perhaps by using  
 the techniques developed to understand fillings of simple singularities as in Ohta--Ono~\cite{OO}. Chen has a  different approach to these questions that is based on extending Seiberg--Witten--Taubes theory to
 the orbifold setting; cf. Chen \cite{Ch}.
 
 Similarly the {\it deformation implies isotopy}  property of $Z$ is very general, and should hold for any blow up that is resolved by some $N$-fold blow up of $\C P^2$.  However, the {\it connectness} property is more delicate, just as it is in the case of the $X_N$.
\end{rmk}

\section{Construction and properties of $M_\ell$.}\labell{s:3}

Most of the first subsection is devoted to the existence proof.  However it also contains Lemma~\ref{le:reg} which, together with Lemma~\ref{le:sing} in \S\ref{ss:uniq}, are the basic ingredients of the
uniqueness proof. The last subsection \S\ref{ss:23}  discusses  the 
cases $\ell=2,3$.

\subsection{Existence.}\labell{ss:exist}

We first  prove the following result.

\begin{prop}\labell{prop:ZZ} For each $\ell$, there is a Hamiltonian $S^1$ manifold $(M^{\le 0}, \Om^-)$ with boundary $(Y^-,\Om^-)$ whose reduced spaces  at level $\ka$ are 
$$
\begin{array}{ll}
\bigl(\C P^2_{1,2,3}, \,\frac{6+\ka}6\tau_{1,2,3}\bigr)&\mbox{when } -6\le\ka < -\ell,\\
(Z,\om_\ka)&\mbox{when } -\ell<\ka\le 0.
\end{array}
$$
\end{prop}

Its proof is based on the following 
 well known lemma. In it, the word  \lq\lq unique"
means unique up to equivariant symplectomorphism.

\begin{lemma}\labell{le:reg} Suppose that $(M,\Om)$ is a (possibly noncompact) Hamiltonian $S^1$-bundle
with proper moment map $H:M\to \R$.  Let $I\subset \R$ be an interval that contains no critical points of $H$ with corresponding family of 
reduced spaces $(V,\om_\ka)$.   Then
the slice $H^{-1}(I)$ is uniquely determined by the family of forms $\om_\ka, \ka\in I,$ on the orbifold $V$.  

Moreover, given any  family of symplectic forms
$\om_\ka, \ka\in I,$ on an orbifold $V$ such 
that $\frac d{d\ka} \om_\ka$ is an integral class whose corresponding orbibundle $S^1\to Y\stackrel{\pi}\to V$ has smooth total space, there is a (unique) corresponding Hamiltonian $S^1$ manifold $(M,\Om)$  with reduced spaces $(V,\om_\ka)$. 
\end{lemma}
\begin{proof}[Sketch of proof]  All statements here, except possibly for the uniqueness, are well known.    If the action on $H^{-1}(I)$ is free,
 then this is proved in \cite[Prop~5.8]{MSI}.
 The key point is to write $\Om$ on $M\equiv Y\times I$ in the form
 $$
 \Om = \pi^*(\om_\ka) + \al_\ka\wedge d\ka,
 $$
 where $\al_\ka$ is a suitable family of connection $1$-forms on the
 circle bundle $\pi:Y\to V$.  The argument in the general case is similar; one simply has to understand the behavior of connection $1$-forms on $S^1$-orbibundles.  For further details, see
 Karshon--Tolman  \cite[\S3]{KT}.  
 \end{proof}

\begin{rmk}\rm  In \cite{KT} Karshon and Tolman work with \lq\lq complexity one" spaces, i.e. manifolds of dimension $2k\ge 4$ with Hamiltonian actions of $T^{k-1}$. In this case, the reduced spaces $V_\ka$ have real dimension $2$ so that their  symplectic structure  is determined by cohomological information --- indeed, just by the Duistermatt--Heckmann measure.  The arguments and definitions
in \cite[\S3]{KT} carry over to the case when the reduced spaces are rigid in the sense of Definition \ref{def:rig}.
For example, Gonzalez  shows in \cite{Gonz} that in the rigid case the total space $(M,\Om)$ depends only on 
the cohomology classes $[\om_\ka], \ka\in I$.  
However, in the general case considered in Lemma \ref{le:reg},
the family of forms  may contain some more information.  Therefore, to use the language of \cite{KT}  one needs to formulate an appropriate redefinition of Karshon--Tolman's concept of a $\Phi-T$-diffeomorphism: besides commuting with the moment map, the induced family of diffeomorphisms on the reduced space $V$ should preserve the family of forms $\om_\ka$.
\end{rmk}

\NI {\bf Proof of Proposition \ref{prop:ZZ}.}
 As illustrated in Figure \ref{fig:7}
we may construct the sublevel set  $(M^{\le-\ell+\eps},\Om^-)$ as a toric manifold.
When translated vertically upwards by $6$ so that its lowest vertex is at the origin and $x_3=\ka +6$, its moment polytope in $\R^3$ is described  by the inequalities:
\begin{gather}\notag
x_1\ge 0,\quad x_2\ge 0,\quad 2x_1+3x_2\le {\ts\frac {x_3}6},\\ \notag x_1+4x_2\ge 1-{\ts \frac 6{\ell}+\frac {x_3}{\ell}},\quad  0\le x_3\le 6-\ell+\eps.
\end{gather}
The slice $(M^{(-\ell,-\ell+\eps)},\Om)$ is a union of circle (orbi-)bundles over the reduced spaces $(Z,\om_\ka)$ and, by Lemma~\ref{le:reg}, may be extended by
attaching  circle (orbi)bundles over $(Z,\om_\ka)$ for $-\ell<\ka\le 0$.
\QED

\begin{figure}[htbp] 
   \centering
   \includegraphics[width=2in]{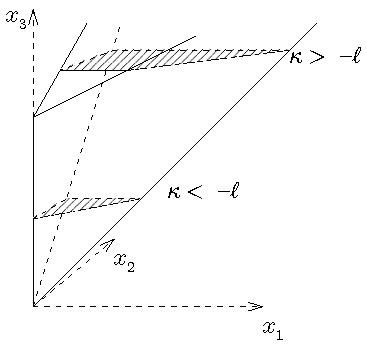} 
   \caption{The moment polytope for  $M^{\le-\ell+\eps}$ showing its two types of 
  critical level.  Here $x_3 = \kappa + 6\ge 0$. }
   \label{fig:7}
\end{figure}

The manifold $(M^{\le 0}, \Om^-)$ realises the sublevel set 
$H^{-1}([-6,0])$.   Denote by $(M^{\ge 0}, \Om^+)$ the 
Hamiltonian $S^1$ manifold that is diffeomorphic to 
$(M^{\le 0}, \Om^-)$ but has the reversed $S^1$ action.  In other words, if $\io$ denotes the identity map and $-id: S^1\to S^1$ takes 
$t$ to $-t$, then there is a commutative diagram
$$
\begin{array}{ccc} S^1\times M^{\le 0} &\stackrel{-id\times \io}\longrightarrow &S^1\times M^{\ge 0}\\
\al^-\downarrow\qquad &&\al^+\downarrow\qquad\\
M^{\le 0}& \stackrel{\io}\longrightarrow& M^{\ge 0},
\end{array}
$$
where $\al^{\pm}$ is the action on $M^{\gtrless 0}$. Then the Hamiltonian generating $\al^+$ is $-H\circ\io$, and
$\io$ induces a symplectomorphism between the reduced 
space of $M^{\le 0}$ at level $\ka\le 0$ and that of $M^{\ge 0}$ at level $-\ka\ge 0.$
\MS

\NI {\bf Proof of Lemma \ref{le:eul}.}   This lemma states that  to
 construct $(M_\ell, \Om)$ it suffices to find a symplectomorphism of $(Z,\om_0)$ that changes the sign of the  Euler class $e(Y)$.  Thus, from such a symplectomorphism $\phi_Z$ we need to find  a way to 
 glue $(M^{\le 0}, \Om^-)$ to $(M^{\ge 0}, \Om^+)$
along their common boundary  in an equivariant and smooth way.  
We shall prove this by applying Lemma~\ref{le:reg}.

In the following argument $\de>0$ is a small constant that may be decreased as needed.
First choose a smooth family $\psi_\ka, \ka\in (-\de,0],$ of diffeomorphisms
from the quotient spaces $H^{-1}(\ka)/S^1$ to $Z$. 
Let $\si_\ka$, $\ka\in (-\de,0],$ be the corresponding smooth family of symplectic forms induced by $\Om^-$. By adjusting $\psi_\ka$ we may assume that $\si_0=\om_0$ and that the 
$\si_\ka$ are standard in some neighborhood 
$\Nn$ of the singular points $\bp$. Then
 the Duistermatt--Heckman formula implies that the $2$-form 
$$
\la : = \frac d{dt}\si_\ka\Bigl|_{t=0}\Bigr.
$$
represents the class $-e(Y)$, and by further adjusting the
$\psi_\ka$ (and deceasing $\de$) we may suppose
that 
$$
\si_\ka = \om_0 +\ka\la,\quad\ka\in (-\de,0].
$$ 
Then Lemma~\ref{le:reg} implies that
a neighborhood of the boundary of $(M^{\le 0},\Om^-)$
is determined by this family $\si_\ka, \ka\in (-\de,0]$.

Now define the $1$-form $\ga$ on $Z$  by the equation 
$\phi_Z^*(\la) = -\la+d\ga$.  Since $H^2(Z) = H^2(Z,\Nn)$,
 we may suppose that $\ga=0$ in $\Nn$.  By a standard Moser argument, there is a smooth family of diffeomorphisms $f_\ka:Z\to Z$  that are the identity in $\Nn$ and such that for small $\de$
 $$
 f_0=id,\quad f_\ka^*(\om+\ka\la) = \om_0+\ka(\la-(\phi_Z^{-1})^*d\ga),\quad \ka\in (-\de,0].
 $$ 
 Now set
 $$
 \om_\ka: = \phi_Z^*\bigl(f_{-\ka}^*(\om_{-\ka})\bigr) = \om_0+\ka\la,\quad \ka\in [0,\de).
 $$
 Again, Lemma~\ref{le:reg} implies that
a neighborhood of the boundary of $(M^{\ge 0},\Om^+)$
is determined by this family $\om_\ka, \ka\in [0,\de)$.
But these two families  fit smoothly together over $(-\de,\de)$.
Hence the result follows by another application of Lemma~\ref{le:reg}.
\QED

\NI {\bf Proof of Proposition \ref{prop:ell} part (i)}.   We must show
 that there is a symplectomeophism $\phi_Z$ as above. To see this,
consider the reduced space  $(Z,\om_0)$ at level  zero.
The form $\Phi_J^*(\om_0)$ is cohomologous to $-K=[\tau]$ (cf. equation (\ref{eq:tau})) since it vanishes on the contracted set $\Cc$
and takes the value $1$ on $E_3,E_k$ by Example \ref{ex:ex}.
(Note that the form $\Phi_J^*(\om_0)$  is degenerate along $\Cc$ and so is not symplectic.) 
Further by equations (\ref{eq:int}) and  (\ref{eq:eY})
\begin{eqnarray}\labell{eq:phie}
&&\Phi_J^*(e(Y))(C_k) = e(Y)({\C P}^1_{1,\ell})=-\ts{\frac 1\ell},\quad \Phi_J^*(e(Y))(C_3)=-{\ts \frac 16}, \\\notag
&&\Phi_J^*(e(Y))(L-E_{123}) =e(Y)({\Phi_J}_*(C_0))=0.
\end{eqnarray}
Hence $\chi_k = -\Phi_J^*(e(Y))$.

Now consider the case $\ell=4$.
 By Lemma~\ref{le:2}
  there is a  diffeomorphism $\psi$ 
of $X_7$ that reverses the sign of  $\Phi_J^*(e(Y))$.  
Denote $\Cc':= \psi(\Cc)$ and  $J': = \psi_*J$, and
let $Z'$ be the image of $(X_7,J')$ under the map $\Phi_{J'}$ that  contracts the curves in $\Cc'$.   
Observe that
 $\Cc'$ is a union of $J'$-holomorphic curves $C'_i$  in  classes  
$\Hat L - \Hat E_{123}, \Hat E_1-\Hat E_2,$ and so on; that is, the
$C_i'$ have the same formulas as do the $C_i$ but with $L, E_j$ replaced by $\Hat L, \Hat E_j$. Further $\psi: X_7\to X_7$ descends to
a diffeomorphism $\psi_Z: Z\to Z'$. Thus we have the middle part of the diagram
\begin{equation}\labell{eq:diag}
\begin{array}{ccccccc} 
&&(X_7,J)&\stackrel{\psi}\to&(X_7,J')&&\\ 
&&\Phi_J\downarrow\;\;&&\;\;\downarrow\Phi_{J'}&&\\
Z^-&\stackrel{f}\leftarrow&Z&\stackrel{\psi_Z}\to&Z'&\stackrel{f'}\rightarrow &Z^+.
\end{array}
\end{equation}

We have constructed $(M^{\le 0}, \Om^-)$ so that there is a 
symplectomorphism  $f: (Z,\om_0)\to (Z^-,\om_0^-)$
such that $(f\circ \Phi_J)^*(e_Z(Y^-))=\chi_7$. Similarly, it follows from  equation (\ref{eq:eps}) that if we allow the classes $\Hat L, \Hat E_j$ (with Poincar\'e duals $\Hat a, \Hat e_i$) to play the roles of $L, E_j$
we can construct a diffeomorphism
$
f': Z'\to (Z^+,\om_0^+)$ such that 
$$ 
\Phi_{J'}^*\,(f')^*(e_Z(Y^+)) =\ts{\frac 1{12}}\bigl(6\Hat a - 2\Hat e_{123} - 3\Hat e_{4567}\bigr)
= -\chi_7.
$$
Denote by $\om_0': = (f')^*(\om_0^+)$ the corresponding symplectic form on $Z'$.   The symplectic forms $\psi_Z^*(\om_0')$ and $\om_0$ on $Z$ pull back to cohomologous forms on $X_7$ and hence are themselves
cohomologous; cf. the proof of Lemma~\ref{le:H2}.
 Therefore Proposition~\ref{prop:Zuniq}  provides a diffeomorphism $g:Z\to Z$ such that $g^*(\psi^*(\om_0')) = \om_0$.
Now take $\phi_Z$ to be the composite:
$$
(Z^-,\om_0^-)\stackrel{f^{-1}}\to (Z,\om_0)\stackrel{\psi_Z\circ g}\to
(Z',\om_0') \stackrel{f'}\to (Z^+, \om_0^+).
$$
This completes the proof for the case $\ell=4$.  The case $\ell=5$ is similar. 
\QED 

\subsection{Uniqueness.}\labell{ss:uniq}

It remains to prove 
the uniqueness statement. We first prove that the germ of $M$ around a critical level is unique.  Then, as in Gonzalez \cite{Gonz}, 
uniqueness will follow from the rigidity of the reduced spaces.

Since this is the only case needed here, we shall suppose that
the critical level $Y_0$ contains a single critical point $q$  with isotropy weights $(a_1,a_2,a_3)$ where
$a_1=-1$, and $a_2,a_3>0$.  As pointed out by Karshon--Tolman \cite{KT}, the difficulty is that the critical level $Y_0$ is not a smooth submanifold near $x_0$, and so its quotient $V_0$ by $S^1$ does not have a natural smooth structure near the image $p$ of $q$ (although
 $V_0$ is diffeomorphic to the reduced spaces $V_\ka, \ka<0,$
  at levels immediately below).
 We therefore  define the smooth structure on $V_0$ near $p$ by 
 choosing an equivariant Darboux chart for the smooth manifold $M$
  at $q$ modelled on the $S^1$ space $(\C^3,0)$ with action and
 moment map
 $$
  (z_1,z_2,z_3)\mapsto (e^{2\pi i a_1}z_1,e^{2\pi i a_2}z_2,e^{2\pi i a_3}z_3),\quad  (z_1,z_2,z_3)\mapsto\sum a_j |z_j|^2.
  $$
  (The symplectic form on $\C^3$ is an appropriate multiple of the standard form.) 
 The equivariant Darboux theorem implies that
 this chart, a baby version of the \lq\lq grommets" of \cite{KT}, is unique up to equivariant isotopy.
Moreover, because $a_1=-1$, the map $\C^2\to H^{-1}(0)\subset \C^3$ given by 
\begin{equation}\labell{eq:slice}
\rho\;:\;(w_2,w_3)\mapsto (\sqrt{a_2|w_2|^2+a_3|w_3|^2}, w_2,w_3)
\end{equation}
meets each orbit  in $H^{-1}(0)$ precisely once and hence provides a 
coordinate chart for a neighborhood of $p$ in $V_0$.  Putting this together with the natural (quotient) smooth structure on $V_0\less p$ 
we get a smooth structure on $V_0$ that is independent of choices. 
Further, the symplectic form $\Om$ on $M$ descends to a symplectic form $\om_0$ on $V_0$ that is again independent of choices.

\begin{lemma}\labell{le:sing}
Suppose we are given two Hamiltonian $S^1$ manifolds
$(M,\Om)$ and $(M',\Om')$ with proper moment maps $H,H'$, each having an isolated critical point
of index $(-1,a_2,a_3)$ at level zero. If
the critical reduced levels $(V_0,\om_0)$ and  $(V_0',\om_0')$
are symplectomorphic, then for some $\eps>0$ 
there is an equivariant symplectomorphism 
$$
\Psi: \bigl(H^{-1}(-\eps,\eps),\Om\bigr)\to \bigl((H')^{-1}(-\eps,\eps),\Om'\bigr).
$$
\end{lemma}

\begin{proof} We shall first lift the given symplectomorphism $\psi_0$ from the critical reduced space $V_0$ to the critical level $Y_0$, and then extend this lift $\Psi_0$  to a symplectomorphism defined on a neighborhood of $Y_0$ by using a modified gradient flow.
 
For the first step,
choose a Darboux chart 
$\chi:U\to U_0$ from  a neighborhood 
$U$ of the fixed point $q\in Y_0\subset M$ to a neighborhood $U_0$
of $0$ in the standard model $\C^3$ described above.
Make a similar choice $\chi':U'\to U_0$ for $M'$.
Isotop the given  symplectomorphism $\psi_0:(V_0,\om_0)\to (V_0',\om_0')$
so that it is the identity in the standard coordinates near the critical points $p, p'$.  More precisely,
with $\rho$ as in equation (\ref{eq:slice}), 
arrange that 
$$
(\rho^{-1}\circ\chi')  \circ\psi_0\circ(\chi^{-1}\circ\rho)\,(w_2,w_3) = (w_2,w_3).
$$
  Then $\psi_0$ lifts to
$(\chi')^{-1}\circ \chi$ in $U\cap Y_0$.  Since there are no fixed points except for $q,q'$, and since $\psi_0$ is a symplectomorphism,
one can show as in Lemma~\ref{le:reg}
that this local lift extends to an equivariant  map $\Psi_0:Y_0\to Y_0'$ such that $\Psi_0^*(\Om') = \Om$.
 
  \begin{figure}[htbp] 
    \centering
    \includegraphics[width=2in]{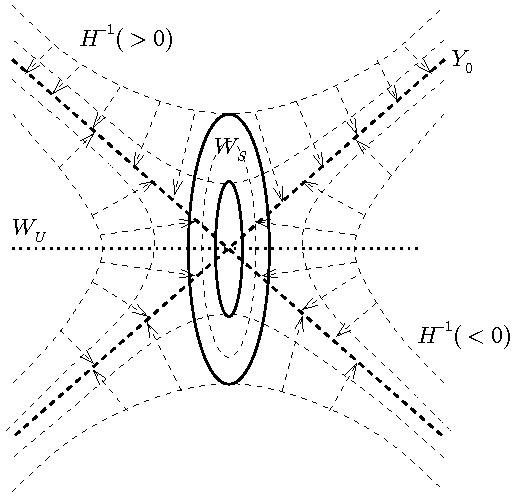} 
    \caption{The maps $\Ff_\ka$ in $U$. The stable manifold $W_S$ is 
    represented by some circles, the critical level $Y_0$ by a pair of  heavy dashed lines, and the unstable manifold $W_U$ by a 
    heavy dotted line.}
    \label{fig:grad}
 \end{figure}

To extend $\Psi_0$ further, 
 choose an invariant $\Om$-compatible almost complex structure $J$ on $M$ that equals the standard  almost complex structure $\chi^*(J_0)$ on $U$, and let $g$ be the corresponding metric. 
Consider the downwards $g$-gradient flow  of $H$ on 
$H^{-1}([0,\eps))$. 
If we choose 
 $\eps>0$ sufficiently small,  we  may suppose that
 $U\cap H^{-1}([0,\eps))$ 
contains all orbits in $H^{-1}([0,\eps))$ 
whose downward flow converges to $q$.  (These  points form the $4$-dimensional stable manifold $W_S$ of $q$ and lie above  
the exceptional divisors in the reduced spaces.)
For each $\ka\in [0,\eps)$ define $\Ff_\ka(x)$ to be
$q$, if $x\in W_S$, and otherwise to be the point where the downward
gradient flow line through $x$ meets $Y_0$. Thus $\Ff_\ka:(Y_\ka,W_S)\to (Y_0,q)$
induces a diffeomorphism $Y_\ka\less W_S\to Y_0\less q$.  Similarly, if $\ka < 0$ define
$\Ff_\ka:(Y_\ka,W_U)\to (Y_0,q)$ by using the upward gradient flow. (Here $W_U$ is the $2$-dimensional unstable manifold of $q$.)

Define $\Ff_\ka'$ 
similarly on $M'$, and then consider the map $\Psi$ that is defined near $Y_0$ by
$$
\Psi(x): =\left\{\begin{array}{ll}(\Ff_{\ka}')^{-1}\circ\Psi_0\circ \Ff_{\ka}(x),& \mbox{ if }
x\in H^{-1}(\ka)\less U\;\mbox{ for some }|\ka|<\eps,\\ (\chi')^{-1}\circ \chi,& \mbox{ if } x\in U. \end{array}\right.
$$
It is easy to check that $\Psi$ is smooth and equivariant.  Moreover, it is a symplectomorphism in $U$ and preserves the symplectic form on   $Y_0$.  Hence a standard Moser argument shows that it can be equivariantly isotoped, by an isotopy that  
is the identity near $q$,   
to an equivariant  symplectomorphism defined near $Y_0$.\end{proof}

\NI {\bf Proof of Proposition \ref{prop:ell} part (ii)}. 
It follows from Theorem~\ref{thm:rigid} that the gluing map 
$\phi_Z$ in diagram (\ref{eq:diag}) is unique up to symplectic isotopy.  Hence $(M,\Om)$ will be unique (up to equivariant symplectomorphism) provided that
the sublevel set $(M^{\le 0},\Om^-)$ is.  
By Lemma 3.7 in \cite{Gonz}, the rigidity of the reduced levels of $M$ implies that, for $I = [-6,-\ell)$ and $I= (-\ell,0]$, 
any two families $\om_\ka, \om_\ka'$, $\ka\in I,$ of symplectic forms with $[\om_\ka]=[\om_\ka']$ for all $\ka$ are  isotopic through such families.  
Hence, Lemma~\ref{le:reg}  imply that the slices $M^{[-6,-\ell)}$ and
$M^{(-\ell,0]}$ have a unique structure.
But the germ of $M$ around the critical level $\ka = -\ell$ is unique by Lemma~\ref{le:sing}.
Therefore the result follows because the maps that glue these pieces together are also unique up isotopy. 
\QED

\subsection{The cases $\ell=2,3$.}\labell{ss:23}

  When $\ell= 2,3$ there is a similar resolution $\Phi_J:X_k\to Z, k=\ell+3$. 
When $k=5$ we may assume that 
the automorphism $\psi:(X_k,\om)\to(X_k,\om)$
takes $L,E_i$ to
\begin{eqnarray*}
\HL&=&3L-E_{1\dots4}-2E_5,\\ 
\HE_i&=& L-E_{j_i5},\;\;\mbox{ where } j_i = 4-i \mbox{ for } 1\le i\le 3, \mbox{ and } j_5=4;\\
\quad \HE_4&=&2L-E_{1\dots5}.
\end{eqnarray*}
Moreover, as in equation (\ref{eq:phie}), $\Phi_J^*([\om_\ka])$ must vanish on $L-E_{123}$  while it is determined on
$E_3$ and $E_5$ by equation (\ref{eq:omZ}).  Hence
\begin{eqnarray*}
\Phi_J^*([\om_\ka]) &=& 
\ts{(1+\frac{\ka}6)\bigl(3a-e_1-e_2-e_3\bigr) - (1+\frac{\ka}2)(e_4+e_5).}
\end{eqnarray*}

\begin{figure}[htbp] 
   \centering
   \includegraphics[width=6in]{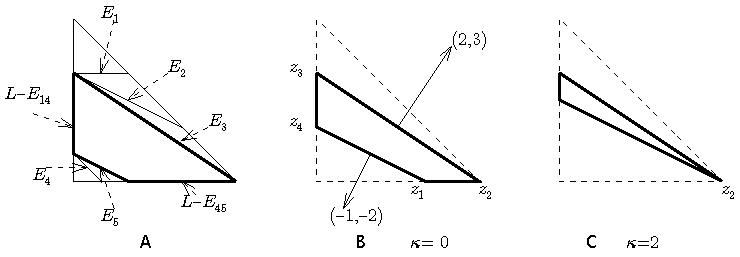} 
   \caption{These are diagrams of the moment polytope of the reduced spaces $(Z,\om_\ka)$ 
   for different $\ka$ in the case $\ell=2$.
    (A) is a schematic representation of $Z$ as a blow up for $\ka$ just larger than $-2$.  (cf. \cite{M});  
  (B) shows the level $\ka=0$ in which 
 $z_1$ is smooth, and  $z_2,z_4$ have order $2$; 
 (C) shows the critical level $\ka = 2$ with $z_2$  smooth.}
   \label{fig:12}
\end{figure}

We have chosen these formulas as in Lemma~\ref{le:J7} so that, 
in the notation introduced there, $\Cc_0=\Cc_0'$ and 
the curves in classes $\HE_3 = L-E_{14}$ and $ \HE_5=L-E_{45}$ 
are represented as well those in classes $E_3, E_5$. (Cf. Figure \ref{fig:12}.)
As $\ka$ increases to the critical level $\ka=2$ 
the area of the curve representing
$\HE_5=L-E_{45}$ 
shrinks to zero.  Hence at this critical level 
the regular point $z_1$ of intersection of $C_{L-E_{45}}$ with
$D_2: = \C P\,\!^1_{12}$ \lq\lq cancels" the singular  point 
$z_2$ on $D_1$.
Thus, going back to the manifold $M_2$, there are isotropy spheres of order $2$ between the critical points $x_4$ and $x_2$ at levels $-6$ and $2$ respectively and between the points $x_3$ and $x_1$ at levels $-2$ and $6$
respectively.
This should be contrasted with the situation
 when $\ell = 4$ or $5$; cf. Figure \ref{fig:1}.  
\begin{figure}[htbp] 
   \centering
   \includegraphics[width=2.5in]{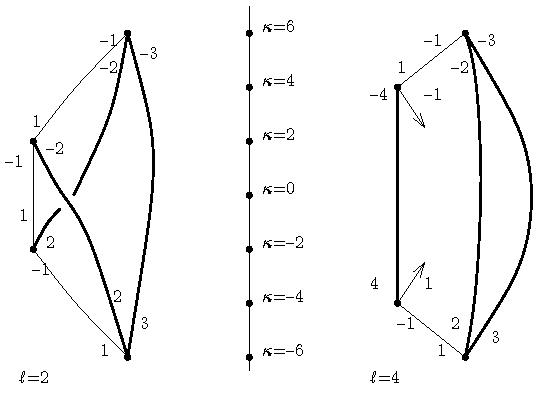} 
   \caption{Diagrams of the $S^1$ invariant gradient spheres
   (i.e. gradient flow lines of  $H$ with respect to an invariant metric)  when $\ell=2,4$. 
   The isotropy spheres are marked with thicker lines. 
     When $\ell=2$ the diagram is a projection of the edges of a $3$-simplex because one can assume that the metric is invariant under a $R^3$ action.}
   \label{fig:1}
\end{figure}

When $\ell=3$ the analogous formulas are
\begin{gather}\notag
\HL = 4L-2E_{123}-E_{456},\quad
\eps_6 =3L-E_{123} -2E_{456},\\\notag
\HE_i = L-E_{jk}, \mbox{ for }\{i,j,k\}=\{1,2,3\},
\quad \HE_i = 2L-\sum_{j\ne i}E_j \mbox{ for }i=4,5,6,\\\notag
\Phi_J^*([\om_\ka]) =\ts{(1+\frac \ka 6)\bigl(3a-e_1-e_2-e_3\bigr) - 
(1+\frac \ka 2)\bigl(e_4+e_5+e_6\bigr)}
\end{gather}
There is no toric model in this case because $(Z,\om_\ka),|\ka|<\ell,$ is constructed by embedding the ellipsoid $\la E(1,\ell)$ into $E(2,3)$ for
$\la< \frac {12}{6+\ell}$ and when $\la>3/\ell$ there is no linear  embedding that does this. However, because $M_3=\Tilde G_\R(2,5)$ supports a Hamiltonian $T^2$-action, there should be an $S^1$-equivariant embedding $\la E(1,3)\to E(2,3)$ for all $\la<\frac 43$.  Note that as in the case $\ell=2$ there is an isotropy sphere (of order $3$ this time) between
the points $x_1$ and $ x_3$ and between $x_2$ and $x_4$.   This was 
noticed by Tolman in \cite{T}; one can check it by calculating the Chern 
class of the isotropy spheres. (Recall that the Chern class of the $S^1$ orbit of 
a gradient flow line from $q$ to $q'$ is the difference  in the sum of the isotropy weights at $q$, $q'$; cf. \cite{MT} for example.)

\section{Complex structures on $M_\ell$.}\labell{ss:com}

Suppose that $J$ is a $\C^*$ invariant complex structure on 
a complex manifold $M$, choose  a K\"ahler metric on $M$ that is invariant under the associated $S^1$ action, and consider the corresponding Hamiltonian function $H$.  Then  the reduced space at a regular level $\ka$  of $H$ can be identified with the quotient
$U_\ka/\C^*$ where $U_\ka$ is the union of all $\C^*$ orbits that intersect the level set $H^{-1}(\ka)$. Since $U_\ka$ 
changes only when $\ka$ passes a critical value, 
 the induced complex structure on the reduced spaces
is constant in each interval $I$ of regular values.  
Next observe that the regular levels $H^{-1}(\ka),\ka\in I,$ fit together
to form a subset $\Ss$  of a holomorphic line orbibundle $\Ll\to Z$ whose fibers 
over the points of $Z$ are (varying) annuli.  (This holds because the fibers of the (holomorphic) projection $H^{-1}(I)\to Z$ support an $S^1$ action that extends to a local $\C^*$ action.)   
Thus $M$ can be considered as a completion of $\Ll\less\Ll_0$, where $\Ll_0$ denotes the zero section of $\Ll$.  Therefore, we will  approach
the  construction of a $\C^*$-invariant complex structure on $M_\ell$
by first finding  a suitable complex structure $J$ on $Z$ and then a suitable orbibundle $\Ll$.   
As usual, we construct $J$ on $Z$ by finding a suitable complex structure on 
the resolution $X_k$.

To do this, it is convenient to change our point of view, thinking of the gluing map $\psi$ of diagram (\ref{eq:diag}) as being the identity map, and the induced map  $L,E_i\mapsto \HL,\HE_i$ (or $\TL,\TE_i$)  as corresponding to a different choice of basis for $H_2(X_k)$.
Moreover the fact that $\psi$ reverses the Euler class  $\chi_k$ translates into the fact that
the formula expressing the  class $\eps_k = PD(\chi_7)$  in terms of the 
first basis should be equal, apart from a sign change, to that
expressing it in terms of the second basis. 

For clarity we shall now denote the second set of classes 
$\HL,\HE_i$ or $\TL,\TE_i$ by $L', E_i'$.  We show below (in Lemma~\ref{le:al}) that it is possible to choose the homology classes $L', E_i'$ so that
the set 
$$
  \Hh'_0:= \{L'-E'_{123},E'_i-E'_{i+1}, i\ne 3,k\}\subset H_2(Z;\Z)
$$ 
coincides with
 $$
 \Hh_0= \{L-E_{123},E_i-E_{i+1}, i\ne 3,k\}.
 $$
Therefore, if these classes have smooth $J$-holomorphic representatives, we can think of the complex space $(Z,J)$ as obtained either by contracting the curves in $\Hh_0$ or those in $\Hh_0'$.  If the complex structure $J$ on $X_7$ also has the property that  the classes 
$E_3, E_k$ and $E_3', E_k'$ have smooth  representatives, then we can identify $(Z,J)$
with the weighted blow up $\C P^2_{1,2,3}\#\ov {\C P}\,\!^2_{1,1,\ell}$ in two ways, identifying the divisors $D_1= \C P^1_{2,3}$ and $D_2=
{{\C P}}\,\!^1(1,\ell)$ either with the images of  $E_3$ and $E_k$ or with the images of $E_3'$ and $E_k'$. We show in Proposition~\ref{prop:com4} below how these ideas lead to a construction of $M_\ell$ as a complex manifold. When there is no danger of confusion we shall sometimes write $E_i$ for the (unique) $J$-holomorphic representative in class $E_i$.

The first step is to find suitable homology classes for $L',E_i'$.  We shall do this first for the case  $k=7$.  In this case, define
\begin{equation}\labell{eq:Cc'}\begin{array}{rcl}
E_i':& = &2L-E_{j4567}, \mbox{ where } (i,j) = (1,3), (2,2), (3,1),\\
E_i': &=& 3L-2E_j-\sum_{m\ne j} E_m, \;\mbox{ where } (i,j) = (4,7),(5,6),(6,5),(7,4)\\
L': &=& 7L-2E_{123}-3E_{4567}, \\ \eps_7': &=& \ts{\frac 1{12}}(6L'-2E_{123}' -3E'_{4567})
 \end{array}
 \end{equation}

 The proof of the next lemma is left to the reader. Note that
 the  somewhat complicated labelling of the $E_i'$ was chosen so that each $E_i'-E_{i+1}' $ in $\Hh_0'$  equals some $E_j-E_{j+1}$ in $\Hh_0$.

  \begin{lemma}\labell{le:al} {\rm (i) }   There is an automorphism $\al$ of $H_2(X_7)$ that takes $L,E_i$ to $L',E_i'$ respectively.  Moreover $\eps_7'=\al(\eps_7) = -\eps_7$.\SSS
 
 \NI{\rm (ii)} $\Hh_0 = \Hh_0'$.
  \end{lemma}

\begin{lemma}\labell{le:J7}    
{\rm (i)}  There is a complex structure $J$ on $X_7$ 
for which the classes $L,E_3,E_7,$ $E_3',E_7'$ as well as those in 
$\Hh_0$ have smooth holomorphic representatives.
\SSS

\NI
{\rm (ii)} This $J$ is unique up to biholomorphism.\SSS
   \end{lemma}

\begin{proof}  
If  $J$ is a complex structure satisfying 
  the hypotheses of (i), then we may  successively blow down $E_3,E_2,E_1$ to a point $p$ and also $E_7,E_6,E_5,E_4$ to a point $q$.  The blow down manifold is diffeomorphic to $\C P^2$ with its unique complex structure.  
  This blow down map takes $L-E_{123}\in \Hh_0$  to a line through $p$ that we shall call $R$. 
Further, it takes the embedded curve $E_7'$  to an immersed
 cubic $T$ with a node at $q$ that is triply tangent to 
 $R$ at $p$. Thus $p$ is a flex point on $T$. Further the curve in class $E_3'$ is taken to a conic $Q$ through $p$ that 
 has a four-fold tangency to $T$ at $q$.  

We claim that, up to projective transformation, there is at most one configuration of this kind.  To see this, note that given $T$ and a choice of flex point $p$, the conic $Q$ is determined by the further choice of a branch $B$ of $T$ at its unique node $q$. But there is a unique choice
of $T,p,q,B$ up to projective transformation.  In fact, 
because all nodal cubics are projectively equivalent,
we may suppose that $T$ is given by the equation $F=0$ where $F = z_3(z_1^2 - z_2^2) -z_1^3$
  with node at $q = [0:0:1]$.  Its three flex points are the points on $T\less \{q\}$ where the Hessian $\left|\frac{\p^2F}{\p z_i\p z_j}\right|$ vanishes, and one can check directly that these are permuted transitively by projective transformations that preserve $T$.  Therefore we may take   
$p=[0:1:0]$ (where $T$ is triply tangent to the line at infinity $z_3=0$).
Note finally the reflection $[z_1:z_2:z_3]\mapsto [-z_1:z_2:z_3]$
interchanges the two branches at $q=[0:0:1]$.

Since there is a unique blowing up process that converts $T$ and $Q$ to curves in $X_7$ in classes $E_7', E_3'$, this proves (ii).
To prove (i) it remains to
 check that there is a configuration of curves $T,Q$ with the required properties.  But given $T$ as above,
    let $Q$ be the unique conic that intersects the branch $B$ to order $4$ at $q$ and also intersects $T$  at  $p$. To see such $Q$ exists, consider the family of conics through the points $p, q=x_4,x_5, x_6,x_7$ where $x_i\in T$, and let the three points $x_5,x_6,x_7,$ converge along the branch $B$  to $q$.  Then the limiting degree $2$ curve intersects $T$ at $q$ to 
    order $5$ and so cannot degenerate into a pair of lines. 
    (The two lines would have to consist of tangents to $T$ at the node $q$, but these are triple tangents and so do not also
     go through $p\in T$.)
  \end{proof}

  When $\ell=5$ we argue similarly, using the formulas:
\begin{gather}\labell{eq:Cc'8}
L': = 16L-5E_{123}-6E_{4\dots8}, \;\;\;
\eps_8' = \ts{\frac 1{30}}(15L'-5E_{123}' -6E'_{4\dots8})\\ \notag
E_i': = 5L-E_{jk}-2E_{i4\dots8}, \;\mbox{ for } \{i,j,k\} =\{1,2,3\},\\\notag
E_i': = 6L-3E_j-2\sum_{m\ne j}E_m, 
\;\mbox{ for } i,j\in \{4,\dots,8\, |\, i+j = 12\}\\\notag
\Hh_0': = \{L'-E'_{123}, E_i'-E'_{i+1}, i\ne 3,8\}
 \end{gather}
It is easy to check the analog of Lemma~\ref{le:al}, while
 Lemma~\ref{le:J8} below replaces  Lemma~\ref{le:J7}.

 \begin{lemma}\labell{le:J8} {\rm (i)} There is a complex structure $J$ on $X_8$ for which the classes $L, E_3,E_8$, $E_3',E_8'$ as well as those in $\Hh_0$ have smooth holomorphic representatives.\SSS
 
 \NI {\rm (ii)} Moreover $J$ is unique up to the choice of a rational parameter $\mu \in \C P^1\less \Ff$, where $\Ff$ is a finite set. 
 \end{lemma}
 \begin{proof}  
 Fix points $p\ne q$ in $\C P^2$, a line $R$ through $p$ but not $q$ and a conic $Q$ through $q$ and not $p$.  
We shall assume that the tangent line to $Q$ from  
$p$ does not go through $q$. In the following construction 
we assume that $p$ and $R$ are fixed but allow $q$ to vary on $Q$.
We shall construct $J=J_q$ on $X_8$ by blowing up 
 $p$  three times and  $q$ five times.
 The blow ups at $p$ are directed by the line $R$ as in the construction of $X_k$ after Lemma~\ref{le:phiZ}.
 Similarly, the five fold blow up at $q$ is directed by $Q$; thus
  the classes $E_4-E_5,\dots, E_7-E_8$ and $2L-E_{45678}$ are all represented by smooth curves.\MS
 
 \NI {\bf Step 1:} {\it For generic $q$, the class $3L-2E_4 - \sum_{m=1,m\ne 4}^8E_m$ is 
not represented in $(X_8,J)$.}

Let $T_q$ be a nodal cubic that is triply tangent to $R$ at $p$ and has node at $q$ with one branch $B_q$ at $q$ tangent to $Q$ to order $4$.
Such  a curve exists by the proof of Lemma~\ref{le:J7}, and is unique because its proper transform $T_q'$ under the first 
$7$ blow ups is an exceptional sphere in the class 
$3L-2E_4 - \sum_{m=1,m\ne 4}^7E_m\in \Ee(X_7)$.    
If $Q'$ is the proper transform of
$Q$ under these blowups then
$$
T_q'\cdot Q' = \bigl(3L-2E_4 - \sum_{m=1,m\ne 4}^7E_m\bigr)\cdot \bigl(2L-E_{4567}\bigr) = 1.
$$
The class $3L-2E_4 - \sum_{m=1,m\ne 4}^8E_m$ is represented in $(X_8,J)$ exactly if the point of intersection $T_q'\cap Q'$ blows down to $q$, that is, exactly if
the branch  $B_q$ is tangent to $Q$ to order $5$
at  $q$.
We claim that this does not happen for generic $q$.  Because the set $\Ff_Q$ of $q\in Q$ for which this happens is algebraic and $Q$ has dimension $1$, it suffices to show that $\Ff_Q\ne Q$.

Suppose that $q\in \Ff_Q$.   Let $Q''\ne Q$ be a conic that is tangent to $Q$ to order $4$ at $q$.  Then $B_q$ is not tangent to $Q''$ to order $5$.  Moreover, there is a projective transformation $\Phi$ of $\C P^2$ that fixes $p,R$ and takes $Q$ to $Q''$.  Let $q_0: = \Phi^{-1}(q)\in Q$.
Then the unique nodal cubic $T_{q_0}$ must coincide with 
$\Phi^{-1}(T_q)$.  Moreover, $T_{q_0}$ is not tangent to $Q$ at $q_0$ to order $5$ because $\Phi(T_{q_0})= T_q$  is not tangent to $\Phi(Q) = Q''$ at $\Phi(q_0) = q$ to order $5$ by construction. 
Hence $q_0\in Q\less \Ff_Q$.\MS
 
 \NI {\bf Step 2:} {\it For generic $q$ the classes $E_3'$ and $E_8'$ have smooth holomorphic representatives.}

Since $E_3', E_8'\in \Ee(X_8)$  they have nontrivial Gromov--Witten invariants and hence have holomorphic representatives for all $q$.  Therefore we just need to check that these representatives  are irreducible.
We will  consider representatives $S'$  for $E_8'$; the argument for $E_3'$ is similar.

 If $S'$ were not smooth it would be the union of components $S'_i$ 
in classes either of the form 
$d_iL - \sum_{k=1}^8 m_{ik}E_k$ with $d_i>0$ and $ m_{ik}\ge0$  or
with $d_i=0$ and  in
the set  $\Hh_0\cup \{E_3,E_8\}$. 
Since these classes  sum to $E_8'$, we must have $\sum d_i = 6$.  Moreover, because the curves in $\Hh_0$ are represented, as is the class $2L-E_{4\dots 8}$ (the proper transform of $Q$), any component that does not lie in $\Hh_0$
 must satisfy the conditions
\begin{gather}\labell{eq:mm}
d_i\ge m_{i1} + m_{i2}+m_{i3},\ \  2d_i\ge \sum_{i=4}^8m_{ik},\\\notag
m_{i1}\ge m_{i2}\ge m_{i3},\;\quad
m_{i4}\ge m_{i5}\ge \dots \ge m_{i8},\\\notag
\sum_k(m_{ik}^2-m_{ik}) \le 2 + d_i^2-3d_i.
\end{gather}
The first conditions above come from positivity of intersections, while the last comes from the fact that these curves 
are rational and so must satisfy the genus zero adjunction inequality
$c_1(S_i)\le 2+ (S_i)^2$. This means that if $d_i=3$ at most one of the 
$m_{ik}$ is $>1$ and that all $m_{ik} \le 2$.  In other words,
the  $m_{ik}$ (listed in decreasing order) are at most $(2,1,\dots,1)$.
Similarly if $d_i=4$ the $m_{ik}$  are at most
$(2,2,2,1,\dots,1)$ or $(3,1,\dots,1)$, while if $d_i = 5$ they are
at most
$(3,3,1,\dots,1)$ or $(3,2,2,2,1,\dots,1)$ or $(2,\dots,2,1,1)$.

Thus, if $d_i=3$ the only permissible class with all 
$m_{ik} \ne 0$ is $(3;1,1,1,2,1,\dots,1)$, where here we have listed the $m_{ik}$ in order of increasing $k$.
But, by Step 2, this element is not  represented. 
(If it were represented we could decompose $E_8'$ as twice this class 
plus  $(E_4-E_5) + \dots + (E_7-E_8) + E_8$.)

Similarly, $E_8'$ would decompose if either of the lines 
$L-E_{145} $ or $L-E_{1234}$ were represented, since $E_8'$ is the sum of 
$6(L-E_{145})$ or $2(L-E_{1234}) + 2(2L-E_{4\dots 8})$  with 
suitable classes from 
$\Hh_0\cup \{E_3,E_8\}$.
However, for generic $q$ these classes are not represented either.  The reader can now check that there are no 
permissible decomposition of $E_8'$.  For example, if one of the curves is 
in class $(4;2,2,2,1,1,1,1,1)$ one could add $(2;0,0,0,1,1,1,1,1)$, but this does not give a large enough coefficient for $E_4$. Also, because of the conditions
$m_{i1}\ge m_{i2}\ge m_{i3}$ and $
m_{i4}\ge \dots \ge m_{i8}$ it does not help to consider classes with
$m_{i1}> m_{i3}=0$ or $m_{i5}> m_{i8}= 0$ since there would have to be other elements in the decomposition with $m_{j3},m_{j8} \ne 0$,
which would make $\sum d_i$ too large. 

These two steps complete the proof of (i).\MS

\NI {\bf Step 3:} {\it Proof of } (ii).

Suppose that $J'$ is any complex structure on $X_8$ for which the classes $E_3'$ and $E_8'$ as well as those in $\Hh_0$ have smooth holomorphic representatives. Then the classes $L-E_{1234}$, $L-E_{145}$ and 
 $3L-2E_4 - \sum_{m=1,m\ne 4}^8E_m$ cannot have  
 holomorphic representatives since they have negative intersection with 
$E_8'$.  Also, any class that is represented by a rational curve
 must satisfy all the conditions in (\ref{eq:mm}) except for $2d_i\ge \sum_{k\ge 4}m_{ik}$ by positivity of intersection with 
 $\Hh_0$.   Hence 
 the  class $2L-E_{4\dots 8}\in \Ee(X_8)$ has a (unique) embedded representative because none of its decompositions 
 satisfy these conditions.   
 
 Since the classes in $\Hh_0$ are represented, there is a blow down map $
 \pi:(X_8,J')\to \C P^2$ that collapses the curves $E_k-E_{k+1}$ for $k\ne 3,8$. Let $p$ be the image of $E_1-E_2$ and $q$ the image of $E_4-E_5$.
 We define   the conic $Q$  to be the blow down of the curve in class 
 $2L-E_{4\dots 8}$.  
Next observe that
 all triples $(p,R,Q)$ 
 consisting of a conic $Q$, a point $p\notin Q$ and a line $ R$  
 through $p$ are projectively equivalent, provided that $R$ 
 is not tangent to $Q$.   Moreover the only way to blow up $\C P^2$ to a complex structure  on $X_8$ for which the curves in $\Hh_0$ as well as $2L-E_{4\dots 8}$ are represented is to perform  repeated blow ups directed  by $R$ at $p$ and  by $Q$ at some point $q\in Q$ as described at the beginning of the proof.
 Hence the only choice in the above construction is the rational parameter
 $q$.    \end{proof}

We shall denote  by $J$ the complex structure on $Z$ induced by the blow down map $(X_k,J)\to Z$ where $J$ is as constructed in the previous two lemmas. 
As explained at the beginning of \S\ref{s:2}, it is also possible 
to construct $Z$ as a toric manifold, with moment polytope as in Figure \ref{fig:2}. Let us denote the corresponding complex structures 
on $Z$ and $X_k$ by $J_T$.  As we pointed out in Remark~\ref{rmk:ZT},
$J_T$ is not equal to the complex structure $J$ in Lemmas~\ref{le:J7} and \ref{le:J8} since in the toric case the blowups at $q$ are also 
directed  by a line (rather than by the conic $Q$)  --- lines can be chosen to be invariant under the group action while conics cannot be. 
Thus the class $L-E_{4\dots k}$ always has a $J_T$ holomorphic representative.   

Both complex structures $J$ and $J_T$ on $Z$ are obtained by blowing up 
the weighted projective space $\C P^2_{123}$ at a point $q$ not on the exceptional divisor $D_1$.  Therefore the divisor $D_2$ is represented 
in both cases. (In fact, $D_1, D_2$ are the images of $E_3$ and $E_8$
respectively.)    Moreover, the complex structures $J$ and $J_T$ coincide on
$Z\less D_2$.

\begin{lemma}\labell{le:JT} For any open neighborhood $U$ of $D_2$ in $Z$ there is a closed neighborhood $V\subset U$
of $D_2$ in $Z$, and a diffeomorphism $f:(Z,J_T)\to (Z,J)$ that is
the identity in $Z\less U$ and on $D_2$ and is
a biholomorphism near $V$.
\end{lemma}
\begin{proof} Because $D_2$ is resolved in $X_k$ by a negative divisor
$D_2'$ (consisting of the curves in classes $E_m-E_{m+1} $ for $4\le m<k$ and $E_k$), there is a unique complex structure  near $D_2'$ in $X_k$ and hence a unique structure near $D_2$ in $Z$. Therefore the identity map on $D_2$ extends to a 
diffeomorphism $g: (V,J_T)\to (Z,J)$ on some neighborhood  $V$ of $D_2$ that is a biholomorphism  onto its image.   Since $g=id$ on $D_2$,
it is easy to find a diffeomorphism of $Z$ that equals $g$ on some shrinking of $V$  and the identity outside  $U$.
\end{proof}

 We are now in a position to prove  the second statement in
  Theorem~\ref{thm:main}.

\begin{prop}\labell{prop:com4}  $M_\ell$ has a $\C^*$-invariant complex
 structure when $\ell=4,5$.  This is unique up to $\C^*$-equivariant biholomorphism when $\ell=4$, 
 and depends on a rational parameter when $\ell=5$.
\end{prop}
\begin{proof}  First consider the case $\ell=4$.  
We will construct $M_4$ to be a holomorphic 
manifold with a holomorphic $S^1$ action.  Then it will automatically have a $\C^*$ action.  $M_4$ will be the union of $3$ pieces corresponding to the three intervals $(-6,-3), (-4+\eps,4-\eps), (3,6)$ of values of the moment map.  We construct the middle slice first.

Denote by $\pi:X_7\to Z$ the map obtained by 
collapsing the curves in $\Hh_0$.
 Let $J$ be the complex structure on $X_7$ constructed in Lemma~\ref{le:J7} and denote also by $J$ the induced complex structure $\pi_*(J)$ on $Z$.  
 Since holomorphic line bundles 
are determined by 
elements of $H^1(\cdot,\C^*)$, it follows from Lemma~\ref{le:H2}
that there is a unique line bundle $\Ll\to (X_7,J)$
 with Euler class $\eps_7$, and that this bundle 
 descends to a holomorphic orbibundle $\ov \Ll$ over $(Z,J)$. Note that the total space of $\ov \Ll$ is smooth because  the boundary
  of the neighborhood $\Vv$ of the curves in $\Hh_0$ is smooth.

Take any Hermitian metric on $\ov \Ll$, pull it back to $\Ll$ and
then, given real valued functions $0<R_1<R_2$ on $Z$ define the slices  $\Ss_0, \ov{\Ss}_0$ by setting
\begin{eqnarray*}
&&\Ss_R = \{(x,v)\in \Ll: R_1(\pi(x))<|v|<R_2(\pi(x))\},\\
&&\ov{\Ss}_R = \{(z,v)\in \ov{\Ll}: R_1(z)<|v|<R_2(z)\}.
\end{eqnarray*}
The manifold $\ov{\Ss}_R$ (for suitably small $R_1$ and large $R_2$)
is the middle part of $(M_4,J)$. Note that it has an $S^1$ action obtained by 
multiplication by $e^{i\theta}$ in the fibers of $\Ll$.

We need to complete $\ov{\Ss}_R$ at both its ends by 
attaching holomorphic manifolds that are diffeomorphic to 
 $M^{< -3}$ and $M^{>3}$.  Let us first consider how to attach  the lower half $M^{< -3}$.  If we think in terms 
 of the $\C^*$ orbits i.e. the fibers of $\Ll$ (forgetting the moment map), we need to compactify 
 $\Ll\less \Ll_0$, replacing the zero section $\Ll_0$ by a copy of $\C^*\cup\{0\}\cup
  \{\infty\}$. Here we can think of $\{0, \infty\}$ as the 
  fixed points at levels $\ka = -6, -4$ respectively.
It is hard to see how to construct such a compactification from $\Ll$ itself.
In particular, when passing $\ka = -4$ we need to collapse the divisor $D_2$ in the zero section $\Ll_0$ to a point and also begin a new $\C^*$ orbit. 

However there is a toric model for this process:
  the discussion before Lemma~\ref{le:Z} implies 
that the set $M^{<-3}$  (considered as  a smooth manifold) can be constructed with a  toric structure 
and so has a  corresponding complex structure $J_T$. Therefore it suffices to
modify $J_T$ so that it matches the complex structure $J$ 
that we already have 
 on  $\ov{\Ss}_R$.   

As above, the slice $(M^{(-4,-3)},J_T): = \bigl(H^{-1}\bigl((-4,-3)\bigr),J_T\bigr)$
can be considered as a subset of a holomorphic orbibundle $\pi: (\Ll_T,J_T)\to (Z,J_T)$ with
Euler class that pulls back to $\eps_7$.   Since the Euler class uniquely determines the bundle,  it suffices to change $J_T$ to $f^*(J)$ in the open set $U\less V$ and make the corresponding modification to the complex structure of $\Ll_T$, where the notation is as in Lemma~\ref{le:JT}.
This defines a new complex structure in 
$M^{(-4,-3)}$ that we shall call $J$.  To see that this extends over  the whole of $M^{<-3}$,
observe that we have not changed the structure near  $\pi^{-1}(D_2)$ 
so that $J$ extends over the critical level $\ka=-4$  
to the sets $M^{(-6+\de, -3)}$ for all $\de>0$.  But the slices
$(M^{(-6+\de,-4)},J)$  are subsets of a holomorphic  orbibundle over 
$(\C P^2_{123},J)$, and $J$ is diffeomorphic to $J_T$ on  $\C P^2_{123}$.  (This follows 
by construction.)  Hence we may identify   $(M^{(-6+\de,-4)},J)$ with an appropriate subset of the canonical bundle over  
$(M^{(-6+\de,-4)},J_T)$ and therefore extend $J$ over the critical level $\ka = -6$ by the toric structure. Note that the resulting complex structure $J$  on
$M^{< -3}$ admits a holomorphic $S^1$ action given by 
multiplying by $e^{i\theta}$ in the fibers of the bundle. 

Now observe that if we choose $R_1$ suitably  we can compactify this end of $\ov{\Ss}_R $ by attaching  $(M^{< -3},J)$.   Again, the union 
$\ov{\Ss}_R \cup M^{< -3}$ has an $S^1$ action.  
Similar remarks apply to the other end. In fact the involution  $(x,v)\mapsto
(x,1/\ov v)$ (where $v\mapsto\ov v$ is complex conjugation)
takes $\Ll\to X_7$ to $\Ll^*\to X_7$.  Since $e(\Ll^*) = -\chi_7 = \chi_7'$
we can repeat the above argument replacing the classes $\Ll,E_i$ by $\Ll', E_i'$.

This constructs the $\C^*$-invariant complex structure on $M_4$.  It is unique up to $\C^*$ equivariant biholomorphism because the complex structure on the reduced space is unique and there are no other choices in the construction.
This completes the proof in the case $\ell=4$. The case $\ell=5$ is almost identical, and is left to the reader.
\end{proof}

\begin{rmk}\labell{rmk:so3}\rm  Since the fixed point data of the $S^1$ actions considered here are symmetric under the inversion $S^1\to S^1:\theta\mapsto -\theta$, the uniqueness 
result implies that there
is an $S^1$-equivariant symplectomorphism of $(M_\ell,\Om)$ that reverses the $S^1$ action. Further, because there is a unique $S^1$ invariant complex structure when $\ell=4$, this map can be taken to be a biholomorphism in this case.   (In fact the existence of such
action reversing  maps is obvious from our construction.)  However, when $\ell=5$ and 
$J_\la, \la \in 
S^2\less\{\mbox{finite set}\},$  is a generic 
 $S^1$ invariant complex structure on $M_5$ then 
  we cannot expect there to be a corresponding biholomorphic map; rather there should be a 
holomorphic  involution $\tau$ of the parameter space 
such that reversing the $S^1$ action takes $(M,J_\mu)$ to $(M,J_{\tau(\mu)})$.  In the case that
 the holomorphic $S^1$ action  $(M,\om)$ extends to 
a holomorphic action of $SO(3)$, then because the $S^1$ action is conjugate to its inverse by an element in $SO(3)$, the corresponding parameter $\mu$ is fixed 
by $\tau$.  Since the  
Mukai--Umemura $3$-fold does have a holomorphic 
$SO(3)$ action,
 it might be interesting to look at it from this point of view.
 One could also try to to analyze it using methods to study Hamiltonian $SO(3)$ actions  such as those developed  
 by R. Chang \cite{Chang}.
\end{rmk}

\end{document}